\documentclass[twoside,12pt,reqno]{amsart}

\date{\today}

\usepackage{accents}
\usepackage[latin1]{inputenc}
\usepackage{pstricks}
\usepackage{enumerate}	
\usepackage{graphicx,color}
\usepackage{caption}
\usepackage[english]{babel}
\usepackage[toc,page]{appendix}
\usepackage{amssymb, mathrsfs, amsfonts, amsmath,amsthm}
\usepackage{amsbsy, hyperref}
\usepackage{latexsym}
\usepackage{booktabs}
\usepackage{tikz}
\usepackage{eurosym} 
\usepackage{newfloat}
\usepackage{hyperref}
\usepackage[alphabetic, msc-links]{amsrefs}

\usepackage{hyperref}
\usepackage{url}

\usepackage{mathtools}

\definecolor{r}{rgb}{.9,0.1,.3}

\numberwithin{equation}{section}

\newtheorem{Theorem}{Theorem}[section]
\newtheorem{Proposition}[Theorem]{Proposition}
\newtheorem{Definition}[Theorem]{Definition}

\newtheorem{Remark}[Theorem]{Remark}
\newtheorem{Lemma}[Theorem]{Lemma}

\newtheorem{remark}[Theorem]{Remark}

\newcommand{\R}{\mathbb{R}}

\newcommand{\n}{\mathbb{N}}

\newcommand{\h}{\mathbb{H}}

\newcommand{\p}{\partial}

\newcommand{\dist}{\operatorname{dist}}
\newcommand{\Div}{\operatorname{div}}

\author[E. S. Gama]{Eddygledson S. Gama}
\address[Gama]{
  Departamento de Matem\'atica,
  Universidade Federal do Cear\'a, Bloco 914, Campus do Pici,
  Fortaleza, Cear\'a, 60455-760, Brazil.
  Departamento de Geometr\'\i{}a y Topolog\'\i{}a,
  Universidad de Granada,
  18071 Granada, Spain.
}
\email{eddygledson@gmail.com}

\author[E. Heinonen]{Esko Heinonen}
\address[Heinonen]{
Departamento de Geometr\'i{}a y Topolog\'\i{}a,
  Universidad de Granada,
  18071 Granada, Spain.
}
\email{ea.heinonen@gmail.com}

\author[J. H. Lira]{Jorge H. De Lira}
\address[Lira]{
  Departamento de Matem\'atica,
  Universidade Federal do Cear\'a, Bloco 914, Campus do Pici,
  Fortaleza, Cear\'a, 60455-760, Brazil.
}

\email{jorge.lira@mat.ufc.br}
\author[F. Mart\'in]{Francisco Mart\'\i{}n}
\address[Mart\'in]{
  Departamento de Geometr\'i{}a y Topolog\'\i{}a,
  Universidad de Granada,
  18071 Granada, Spain.
}
\email{fmartin@ugr.es}

\thanks{
E. S. Gama is supported by Coordena\c{c}\~ao de
Aperfei\c{c}oamento de Pessoal de N\'{i}vel Superior - Brasil CAPES/PDSE/88881.132464/2016-01. E. Heinonen supported by a grant from the Finnish Academy of Science and Letters. J. H. de Lira is supported by PRONEX/FUNCAP/CNPq PR2-0101-00089.01.00-15. and
CNPq/Edital Universal 409689/2016-5. F. Mart\'in is partially supported by the MINECO/FEDER grant MTM2017-89677-P and  by the
Leverhulme Trust grant IN-2016-019.}


\title[Translating horizontal graphs]{Jenkins-Serrin problem for translating horizontal graphs in $M\times\R$}

\begin{document}
\maketitle
\begin{abstract} 
	We prove the existence of horizontal Jenkins-Serrin graphs that are translating solitons of the mean curvature flow in Riemannian product manifolds $M\times\R$. Moreover, we give examples of these graphs in the cases of $\R^3$ and $\mathbb{H}^2\times\R$.
\end{abstract}

\section{Introduction}

Let $M^n$ be a Riemannian manifold and $\Omega\subset M$ be a domain (not necessarily bounded) with piecewise smooth boundary. Assume that the boundary can be composed as $\p\Omega = \Gamma_0 \cup \Gamma_1 \cup \Gamma_2 \cup N$, where the sets $\Gamma_i$ are open, disjoint and smooth, and the set $N$ is closed subset so that the $(n-1)$-dimensional Hausdorff measure $\mathcal{H}(N)=0$. 
Then a classical problem is to find the sufficient and necessary conditions for the existence of prescribed mean curvature surfaces with possibly infinite boundary data, or more precisely, to solve the Dirichlet problem
\begin{equation}\label{Div. equation}
\begin{cases}
	\Div\left(\frac{\nabla u}{\sqrt[]{1+|\nabla u|^2}}\right) = H(x), &\text{in }\ \Omega; \\
	u=u_0,  &\text{on } \ \Gamma_0; \\
	u=+\infty, &\text{on } \ \Gamma_1;\\
	u=-\infty, &\text{on } \ \Gamma_2,
\end{cases}
\end{equation}
where $H \colon M\to\R$ is a Lipschitz function and $u_0 \colon \Gamma_0\to\R$ is a continuous function.

The most famous example of solutions of \eqref{Div. equation} in $\R^2$ with $\Omega = \left[ -\pi/2, \pi/2 \right] \times \left[ -\pi/2, \pi/2 \right]$ was given by H. Scherk in 1834. Namely, he proved that the function $u=\log (\cos x / \cos y )$ is a solution of \eqref{Div. equation} with $\Gamma_0=\varnothing$ and $H\equiv0$, obtaining the Scherk's minimal surface. More than a hundred years later H. Jenkins and J. Serrin \cite{Jenkins-Serrin} found the necessary and sufficient conditions for the existence of solutions of \eqref{Div. equation} in $\R^2$ with $H\equiv0$. Their clever idea was to use a part of Scherk's example as a barrier for sequences of solutions, and so they related the existence of solutions of  \eqref{Div. equation} with algebraic conditions involving the length of ``admissible polygons" in the domain.  Later the Dirichlet problem \eqref{Div. equation} became known as the Jenkins-Serrin problem.

J. Spruck \cite{Spruck-Infinite} extended the results of Jenkins and Serrin for constant mean curvature surfaces in $\R^2$ and gave local existence for general domain in $\R^n$. On the other hand, using different methods, U. Massari \cite{Massari} proved that it is possible to study also the case of prescribed mean curvature, and extended the results for solutions of \eqref{Div. equation} when $H$ is not a constant in $\R^n$ but satisfies some structural conditions. His idea was to replace the algebraic conditions involving the length of admissible polygons by conditions on certain functionals defined on Caccioppoli sets. See also E. Giusti \cite{Giusti} for a more detailed explanation of Massari's ideas in the case $H\equiv0$. 

More recently the Jenkins-Serrin problem has been studied in many different settings and we mention here some of the most closely related works to our current paper. B. Nelli and H. Rosenberg \cite{Nelli-Rosenberg} studied the existence of minimal graphs in $\mathbb{H}^2\times\R$ for domains $\Omega\subset \h^2$, A. L. Pinheiro \cite{Pinheiro} extended their work into $M^2\times\R$ with $M^2$ a general Riemannian surface and $\Omega \subset M^2$ geodesically convex domain, and M. H. Nguyen \cite{Nguyen} extended these results further into the case of $\text{Sol}_3$. Possibly unbounded domains were studied in \cite{Mazet-Rodriguez-Rosenberg}, and for other works, also with the CMC case, one should see \cite{Collin-Rosenberg, Folha-Melo, Folha-Rosenberg,  Galvez-Rosenberg, Hauswirth-Rosenberg-Spruck, Klaser-Menezes, Younes}. Another interesting paper, that is in slightly different setting, is \cite{eichmair-metzger} where M. Eichmair and J. Metzger studied the existence of Scherk type solutions for the Jang's equation on Riemannian manifolds with dimension at most 7. More detailed description of earlier results can be found from \cite{GHLM}, where the authors obtained also a Jenkins-Serrin type result for translating graphs.

In this paper we are considering the Jenkins-Serrin problem for translating horizontal  graphs in $M^2\times\R$, where $M^2$ is a $2$-dimensional Riemannian manifold possessing a non-singular Killing vector field $Z$. These translating horizontal graphs are translating (into the vertical direction $\partial_t$) solitons of the mean curvature flow that can be considered as graphs of a function $u\colon \mathbb{P} \to \R$ defined on a vertical plane $\mathbb{P} \subset M^2\times\R$. Here the plane $\mathbb{P}$ will be a totally geodesic leaf from the orthogonal distribution of the Killing field $Z$, and therefore the original Riemannian product can be written as a warped product $M^2\times\R = \mathbb{P}\times_\rho \R$, where $\rho$ is a smooth warping function. 

Due to a result by T. Ilmanen, the translating solitons can be considered as minimal surfaces in $M\times\R$ equipped with a conformally changed metric. This, and the warped structure, leads to a modification of the PDE in \eqref{Div. equation}. More precisely, we are considering solutions to the equation
	\begin{equation}\label{JS-sol-eq}
	\Div_{\mathbb{P}} \left( f^2 \frac{\nabla u}{\sqrt{1 + f^2 |\nabla u|^2}} \right) = 0,
	\end{equation}
where $\Div_{\mathbb{P}}$ and the gradient are taken with respect to the Ilmanen's metric restricted to the vertical plane $\mathbb{P}$, and the function $f$ depends on the warping function $\rho$ and on the Ilmanen's conformal change. This structure causes some difficulties for the extension of the Jenkins-Serrin theory, since after the conformal change of the metric, the plane $\mathbb{P}$ is not complete. However, these problems can be overcome by using some ideas that were developed in \cite{eichmair-metzger} and  \cite{Hoffman}.

It is worth to point out that equation (\ref{JS-sol-eq}) is formally equivalent to the minimal surface equation for graphs in warped product spaces with warping function $f$. This suggests that the results here can be easily adapted to this setting. In particular, the main steps in our construction could provide Jenkins-Serrin type theorems in warped product spaces.

This paper is structured as follows. In Section \ref{Preliminaries} we introduce the notation that will be used, explain the setting in detail and justify the PDE \eqref{JS-sol-eq}. The local theory, that is needed to extend the classical Jenkins-Serrin ideas, is developed in Section \ref{Local J-S}, and in Section \eqref{E-J-S} we prove our main result, Theorem \ref{Existence}, that solves the Dirichlet problem \eqref{Div. equation} with the PDE \eqref{JS-sol-eq}. Uniqueness of these solutions is obtained in Theorem \ref{thm-uniq}. The paper is finished by Section \ref{Examples}, where we give examples of domains in $\R^3$ and $\h^2\times\R$ that satisfy the conditions of the existence result Theorem \ref{Existence}.

\section{Preliminaries} \label{Preliminaries}
In this section we will introduce the principal concepts that we are going to use along of the paper. The notation follows mainly \cite{DHL} and \cite{Lira}.

\subsection{Translating solitons} Let $M^2$ be a complete Riemannian surface and $Z$ a non-singular Killing vector field in $M^2.$ We can see $Z$ as a Killing vector field in $M^2\times\R$ by horizontal lift of $Z$, $Z(p,t)\coloneqq Z(p)$ for all $(p,t)\in M^2\times\R$. Clearly this is a Killing field in $M^2\times\R$ endowed with product metric $g_0\coloneqq \sigma+dt^2,$  where $\sigma$ is a Riemannian metric in $M^2.$

Let $\mathbb{P}$ be a fixed totally geodesic leaf of the orthogonal distribution associated to $Z$ in $M^2\times\R$. Since $Z$ is a horizontal lifting of a Killing field in $M^2,$ then $\mathbb{P}=\Gamma\times\R,$ where $\Gamma$ is a geodesic associated to orthogonal distribution of $Z$ in $M$. Let $\Psi\colon \mathbb{P}\times\R\to M\times\R$ be the flow generated by $Z$. Using this flow we can get local coordinates as follows. If we take any local coordinate $x$ in $\Gamma$, then we obtain coordinates for a point $p\in M^2\times\R$ using the flow of $Z$, i.e., 
$p=\Psi((x,t),s)$. Therefore $(x,t,s)$ is a local coordinate for $\mathbb{P}\times \R =M^2\times\R$. Moreover, the corresponding coordinate vector fields are given by
\begin{align*}
\partial_s(x,t,s) &= Z(\Psi((x,t),s)); \\ 
\partial_t(x,t,s) &= \Psi_*((x,t),s) \partial_t(x,t); \\ 
\partial_x(x,t,s) &= \Psi_{*}((x,t),s) \partial_x(x,t).
\end{align*}

Using these coordinate vector fields we see that the components of the product metric are given by
\begin{align*}
g_{11} &= \langle\partial_s,\partial_s\rangle=:\rho^2(x), \quad  g_{12}=\langle\partial_s, \partial_x\rangle=0,\quad  g_{13}=\langle\partial_s,\partial_t\rangle=0 \\ 
 g_{22} &= \langle\partial_x,\partial_x\rangle=: \varphi^2(x), \quad  g_{23}=\langle\partial_t,\partial_x\rangle=0,\ \operatorname{and}\ g_{33}=\langle\partial_t,\partial_t\rangle=1.
\end{align*}
Therefore
	\[
	g_0=\varphi^2(x)dx^2+\rho^2(x)ds^2+ dt^2
	\]
and we see that $M^2\times\R$ is locally a warped product. Motived by this, we consider from now on that $M^2=S\times_{\rho}\R$, where the first factor $S$ may be either $\mathbb{S}^1$ or $\R$ endowed with a Riemannian metric $\varphi^2(x)dx^2$  and $\rho$ is any positive smooth function in $S$. With this convention $\mathbb{P}=S\times\R,$ with Riemannian metric $h_0 \coloneqq \varphi^2(x)dx^2+ dt^2$ and $M^2\times\R=\mathbb{P}\times_{\rho}\R$. Moreover, a horizontal graph in $M\times\R$ over a domain $\Omega\subset \mathbb{P}$ means a surface $\Sigma\subset M^2\times\R$ given by
\[
\Sigma=\{\Psi(x,t,u(x,t))\in \mathbb{P}\times_{\rho}\R \,(=M^2\times\R) \colon (x,t)\in\Omega \},
\]
where $u\colon\Omega\to\R$ is a smooth function. Sometimes, to simplify the notation, we will write also ${\rm Graph}[u]$ to mean the horizontal graph of $u$.
\begin{Remark}
The horizontal graphs, that we are considering in this paper, are graphs in the direction of the Killing field $\partial_s$. However, we are representing them as ``vertical" graphs since they are graphs in $\mathbb{P} \times_{\rho}\R$ ``over" a domain in $\mathbb{P}$. Therefore the last coordinate is the coordinate associated to the flow lines of $\partial_s$. 
Moreover, for us a horizontal line means a flow line of the vector field $\partial_s$, i.e., $\Psi(q,\R) =  \{\Psi(q,s) \in \mathbb{P}\times_{\rho}\R \, (=M\times\R) \colon q\in\mathbb{P},\, s\in\R\}$. With abuse of notation, we will also write $\Psi(q,\R) = \{q\} \times \R$.
\end{Remark}

We say that a surface $\Sigma$ (not necessarily a graph) in $M^2\times\mathbb{R}$ is a translating soliton with respect to the parallel vector field $X=\partial_t$  (with translation speed $c\in\mathbb{R}$) if 
\begin{equation*}
\textbf{H}=c\,X^{\perp},
\end{equation*}
where $\textbf{H}$ is the mean curvature vector field of $\Sigma$ and $\perp$ indicates the projection onto the normal bundle of $\Sigma$. In particular, if N is a normal vector field along $\Sigma$, then we have
\begin{equation}\label{TS}
{\rm H}=c\langle X, N \rangle,
\end{equation}
where $\langle \cdot, \cdot\rangle=\sigma+dt^2$ denotes the Riemannian product metric in $M^2\times\mathbb{R}$.
 
In \cite{Ilmanen} T. Ilmanen proved the following useful relation between the translating solitons and minimal surfaces with respect to the so-called Ilmanen's metric (conformal change of the metric).
\begin{Lemma}[T. Ilmanen]\label{Ilmanen}
Translating solitons with translation speed $c\in\mathbb{R}$, and translating direction $\partial_t$, are minimal hypersurfaces in the product $M\times\R$ with  respect to the Ilmanen's metric $g_c=e^{ct}g_0$.
\end{Lemma}
Since we are considering the Riemannian metric $g_{0}=h_0+\rho^2ds^2=\varphi^2(x)dx^2+dt^2+\rho^2(x)ds^2$, the conformal change of Ilmanen can be written as 
\[
g_c=e^{ct}(\varphi^2(x)dx^2+dt^2+\rho^2(x)ds^2) \eqqcolon h_c+e^{ct}\rho^2(x)ds^2,
\]
where $h_c$  denotes the restriction of Ilmanen's metric $g_c$ to $\mathbb{P}$. Note that $g_c$ is still a warped metric. From now on we will always consider the metric $h_c$ in $\mathbb{P}$ and the metric $g_c$ in $M\times\R.$ Also, to simplify the notation we will denote by $f\colon M\times\R\to\R$ the function $f(x,t)= e^{\frac{c}{2}t}\rho(x)$.

\begin{Remark}\label{Ilmanen's metric.}
Notice that the Ilmanen's metric is not a complete metric in $M\times\R$ but we will need that $(M\times\R,g_0)$ is complete. 
\end{Remark}

When a surface $\Sigma$ in $M\times\R$ is a horizontal graph of a function $u \colon \Omega \subset \mathbb{P} \to\R$, then $\Sigma$ can be oriented by the unit normal vector field
\[
N=\frac{1}{f}\frac{\partial_s}{W} - f\frac{\nabla u}{W},
\]
where, to simplify the notation, we denote by $\nabla u$ the translation $\Psi_*\nabla u$ from $x\in \Omega$ to the point $\Psi(x, u(x))\in \Sigma$. Moreover, from \eqref{TS} we can check that $u$ satisfies the partial differential equation
\begin{equation}\label{soliton}
	\Div_\mathbb{P}\left(f^2\frac{\nabla u}{W}\right)=0 \quad \operatorname{in} \ \Omega,
\end{equation}
where $W\coloneqq \sqrt{1+f^2h_c(\nabla u,\nabla u)}$, and the gradient and divergence are taken with respect to the metric $h_c$ in $\mathbb{P}$. Indeed, observe that 
\[
 N =\frac{1}{f}\frac{\partial_s}{W}-f\frac{\nabla u}{W}=\frac{\partial_s}{f^2W_f}-\frac{\nabla u}{W_f},
\]
where $fW_f=W.$ Since ${\rm Graph}[u]$ is a minimal surface in $(M\times\R,g_c)$ we have
\begin{align*}
0&=\Div_\Sigma\left( N \right)=\Div_{M\times\R}\left({\rm N}\right)=\Div_{M\times\R}\left(\frac{\partial_s}{f^2W_f}-\frac{\bar{\nabla} u}{W_f}\right)\\
&= \Div_{M\times\R}\left(\frac{\partial_s}{f^2W_f}\right)-\Div_{M\times\R}\left(\frac{\bar{\nabla} u}{W_f}\right)\\
&= \frac{1}{f^2W_f}\Div_{M\times\R}\left(\partial_s\right)+g_c\left(\nabla\left(\frac{1}{f^2W_f}\right),\partial_s\right)-\Div_{M\times\R}\left(\frac{\bar{\nabla} u}{W_f}\right)\\
&=-\Div_{M\times\R}\left(\frac{\bar{\nabla} u}{W_f}\right)=-\frac{1}{f^2}g_c\left(\bar{\nabla}_{\partial_s}\left(\frac{\bar{\nabla} u}{W_f}\right),\partial_s\right)-\Div_{\mathbb{P}}\left(\frac{\nabla u}{W_f}\right)\\
&=\frac{1}{f^2}g_c\left(\bar{\nabla}_{\partial_s}\partial_s,\frac{\bar{\nabla} u}{W_f}\right)-\Div_{\mathbb{P}}\left(\frac{\nabla u}{W_f}\right),
\end{align*}
where $\nabla$ denotes the connection in $(\mathbb{P},h_c=g_c|\mathbb{P})$, $\bar\nabla$ the connection in $(M\times\R,g_c)$ and we have used the fact that $\mathbb{P}$ is totally geodesic in $M\times\R.$ As $g_c(\partial_s,\partial_s)=f^2$, one obtains
\[
\bar{\nabla}_{\partial_s}\partial_s=-f\bar{\nabla}f
\]
and hence, using again that $\mathbb{P}$ is totally geodesic, we conclude that
\[
0=h_c\left(\frac{\nabla f}{f},\frac{\nabla u}{W_f}\right)+\Div_{\mathbb{P}}\left(\frac{\nabla u}{W_f}\right)=\frac{1}{f}\Div_{\mathbb{P}}\left(f\frac{\nabla u}{W_f}\right)=\frac{1}{f}\Div_{\mathbb{P}}\left(f^2\frac{\nabla u}{W}.\right)
\]


\subsection{Conformal geometry of $M\times\R$} 
Here we will collect some computations about the conformal structure of $(M\times\R,g_c)$. Recall that we have $f(x,t) = e^{\frac{c}{2}t}\rho(x)$ and $\nabla$ denotes the connection in $(\mathbb{P},h_c=g_c|\mathbb{P})$, $\bar\nabla$ the connection in $(M\times\R,g_c)$, and we use the convention that the mean curvature is just the trace of the second fundamental form.

Let $\gamma \colon [0,1]\to M\times\R$ be a parametrized curve in $M\times\R$. We define the $f$-length of $\gamma$ by
\begin{equation}\label{f-Length}
	{\rm L}_f[\gamma]=\int_0^1f(\gamma(r)) \sqrt{g_c(\gamma'(r),\gamma'(r))_{\gamma(r)}} \, dr.
\end{equation}

We will work with a special type of curves that are defined in the following way.

\begin{Definition}
Let $\gamma$ be a curve in $M\times\R$ . We say that $\gamma$ is an $f$-geodesic if 
\begin{equation}\label{f-geodesic equation}
\bar{\nabla}_{\partial_r}\gamma'=g_c(\gamma',\gamma')\frac{\bar{\nabla} f}{f}-2g_c\left(\frac{\bar{\nabla} f}{f},\gamma'\right)\gamma',
\end{equation}
where $\bar{\nabla}_{\partial_r}\gamma'$ denotes the covariant derivative of $\gamma'$ along $\gamma$ with respect to $g_c$.
\end{Definition}

\begin{Definition}[$f$-curvature]
Let $\gamma$ be a curve in $\mathbb{P}$. The (scalar) $f$-curvature of $\gamma$ is
	\begin{equation}\label{f-geodesic curvature}
		{\rm k}_f[\gamma]:={\rm k}_{h_c}[\gamma]-h_c\left(\frac{\nabla f}{f},N\right),
	\end{equation}
where ${\rm k}_{h_c}[\gamma]$ denotes the geodesic curvature of $\gamma$ in $(\mathbb{P},h_c)$ and $N\in T\mathbb{P}$ denotes the unit normal along $\gamma$.
\end{Definition}

Before proceeding, we remark some properties of $f$-geodesics that will be used later. 
\begin{Remark}\label{f-prop.}
\begin{enumerate}
\item Let $\gamma$ be a curve in $\mathbb{P}$. Consider the surface (Killing cylinder) $\gamma\times\R = \Psi(\gamma,\R)\subset M\times\R$ ruled by the flow lines of $\partial_s$ passing through $\gamma$. It is straightforward to see that
\begin{equation*}
{\rm k}_f[\gamma]={\rm H}_{\gamma\times\R},
\end{equation*}
where ${\rm H}_{\gamma\times\R}$ denotes the mean curvature of $\gamma\times\R$ in $(M\times\R,g_c)$.
Namely,  $\{\Psi_*\gamma', \partial_s/f\}$ is an orthonormal frame and therefore the mean curvature vector field of $\gamma\times \mathbb{R}$ is given by
	\begin{align*}
	\left( \bar\nabla_{\Psi_*\gamma'} \Psi_*\gamma' + \bar\nabla_{\partial_s/f}( \partial_s/f)   \right)^\perp
	= \left( \bar\nabla_{\Psi_*\gamma'} \Psi_*\gamma' - \frac{\bar\nabla f}{f}  \right)^\perp.
	\end{align*}
Now, if $N\in T\mathbb{P}$ is unit normal to $\gamma$.  then $\Psi_* N$ is unit normal to $\gamma\times\R$ and the scalar mean curvature is
	\begin{align*}
	{\rm H}_{\gamma\times\R} &= g_c \left( \bar\nabla_{\Psi_*\gamma'} \Psi_*\gamma', \Psi_*N \right) - g_c\left( \frac{\bar\nabla f}{f}, \Psi_*N \right)  \\
	&= g_c(\nabla_{\gamma'} \gamma',N) - g_c\left( \frac{\nabla f}{f}, N \right). 
	\end{align*}
Therefore there is a correspondence between $f$-geodesics and minimal cylinders in $M\times\R$. 

\item From \eqref{f-geodesic equation} we see that a curve $\gamma$ in $\mathbb{P}$ is an $f$-geodesic in $\mathbb{P}$ if and only if $\gamma$ is an $f$-geodesic in $M\times\R$.

\item\label{f-Area} Let $\gamma$ be a curve in $\mathbb{P}$ and consider the Killing rectangle over $\gamma$, with height $h$, defined by $\gamma\times[0,h] \coloneqq \Psi(\gamma,[0,h])=\{\Psi(p,s)\in \mathbb{P}\times_\rho\R \colon p\in \gamma, \, s\in[0,h]\}$, where $h>0$. Then we have
\begin{equation*}
\operatorname{Area}[\gamma\times[0,h]]=\int_{0}^{1}\int_{0}^{h}f(\gamma(r))  \sqrt{h_c(\gamma'(r),\gamma'(r))} \, drdz = h {\rm L}_f[\gamma].
\end{equation*}
Note that the length of a segment $\{ \Psi((x,t),s) \colon s\in[0,h] \}$ of a flow line through the point $(x,t)\in\mathbb{P}$ is given by $h f(x,t)$.
\end{enumerate}
\end{Remark}

In order to guarantee the existence of (at least short) $f$-geodesics, we consider the following conformal change of the metric. Let $\sigma_c \coloneqq f^2 g_c = e^{2\log f} g_c$ and denote by $\tilde\nabla$ the Riemannian connection in $M\times\R$ with the metric $\sigma_c$. Since, under the conformal change, the connection changes by
\[
	\tilde{\nabla}_YX=\bar{\nabla}_YX+g_c\left(X,\frac{\bar{\nabla} f}{f}\right)Y+g_c\left(Y,\frac{\bar{\nabla} f}{f}\right)X-g_c\left(X,Y\right)\frac{\bar{\nabla} f}{f}
\]
we conclude from \eqref{f-geodesic equation} that $f$-geodesics are geodesics in $(M\times\R,\sigma_c)$.
In particular, we can use the classical theory about the existence of geodesics and exponential mapping, and we have the following two results that will be used later.

\begin{Proposition}\label{minimize charact.}
The $f$-geodesics are critical points of the $f$-length with respect to proper variations. Moreover, the $f$-geodesics are minimizers of the $f$-length.
\end{Proposition}

\begin{Proposition}\label{totally f-geodesic nei.}
Given any point $p\in\mathbb{P}$, then there exists a neighbourhood $U\ni p$ satisfying the following property: Given any $q_1, q_2\in \overline{U}$ then there is a unique $f$-geodesic joining $q_1$ and $q_2$, and the interior of this $f$-geodesic lies in $U$.
\end{Proposition}

\begin{Remark}
The neighbourhood given by Proposition \ref{totally f-geodesic nei.} will be called geodesically $f$-convex neighbourhood.
\end{Remark}

\section{Local Jenkins-Serrin theory}\label{Local J-S}

In this section we study the local existence of Jenkins-Serrin graphs and the behaviour of sequences of solutions of \eqref{soliton}.

\subsection{Local existence} We prove the existence of solutions over admissible domains, that are defined as follows.

\begin{Definition}[Admissible domain]
Let $\Omega \subset \mathbb{P}$ be a precompact domain. We say that $\Omega$ is an admissible domain if \(\partial\Omega\) is a union of $f$-geodesic arcs $A_1,\ldots,A_s,$ $B_1\ldots,B_r$, $f$-convex arcs $C_1,\ldots,C_t$, and the end points of these arcs and no two arcs $A_i$ and no two arcs $B_i$ have a common endpoint. Each of those arcs is called an {\rm edge} and  their common endpoints are called {\rm vertices}.
\end{Definition}

\begin{Definition}[Admissible polygon]
Let $\Omega$ be an admissible domain. We say that $\mathcal{P}$ is an admissible polygon if $\mathcal{P}\subset\overline{\Omega}$, the boundary of $\mathcal{P}$ is formed by edges of $\p\Omega$ and $f$-geodesic segments, and the vertices of $\mathcal{P}$ are chosen among the vertices of $\Omega.$
\end{Definition}

Suppose now that $\Omega \subset\mathbb{P}$ is an admissible domain with $\p\Omega = \cup_i J_i$, where the family $\{J_i\}\subset\p\Omega$ is a closed cover of $\p\Omega$ and satisfies  $J_i\cap J_{i+1}=\alpha_i$ for all $i\in\{1,v-1\}$, and $J_v\cap J_1=\alpha_v$, where $\{\alpha_i\}$ denotes the set of endpoints of the arcs $J_i$. Let $c=\{c_i \colon J_i\to\R\}$ be a family of bounded continuous functions. Consider the curve $\gamma_c \subset \partial\Omega\times\R = \Psi(\partial\Omega,\R)$ given by $\gamma_c(x)=\Psi(x,c_i(x))$ if $x\in \operatorname{int}J_i$ and $\gamma_c$ is a horizontal line joining $\Psi(\alpha_i,c_i(\alpha_i))$ and $\Psi(\alpha_i,c_{i+1}(\alpha_i))$ if $x=\alpha_i$. Then, as we will see, it is always possible to get a solution of \eqref{soliton} with boundary data $\gamma_c$ over a geodesically $f$-convex domain. Here boundary data $\gamma_c$ means that the solution equals to $c_i$ on $J_i$.

\begin{Theorem}[Local Existence]\label{Local Existence}
Let $\Omega$ be a geodesically $f$-convex domain which is also an admissible domain in $\mathbb{P}$ as above. Let $c=\{c_i \colon J_i\to\R\}$ be a family of bounded continuous functions and $\gamma_c$ the curve associated to $c$. Then there exists a unique solution of \eqref{soliton} with boundary data $\gamma_c$.
\end{Theorem}

\begin{proof}
The proof is essentially the same as of Pinheiro \cite[Theorem 1.2]{Pinheiro} and Nguyen \cite[Theorem 3.10]{Nguyen}. Namely, by Remark \ref{f-prop.} (1) the domain bounded by the flow lines of $\partial_s$ passing through $\partial\Omega$ is mean convex and therefore we can solve the Plateau problem with boundary data $\gamma_c$. To show that the obtained solution is a graph over the domain $\Omega\subset\mathbb{P}$, we can use the same strategy as in \cite{Pinheiro} since, by Remark \ref{f-prop.} (1), the surface $\gamma\times\R$ is a minimal surface when $\gamma$ is an $f$-geodesic.
\end{proof}

\begin{Remark}
Theorem \ref{Local Existence} is not in contradiction with the non-existence result of F. Chini and N. M. M{\o}ller \cite[Proposition 30]{Chini-Moller} because the cylinder over the domain considered by them is not $f$-convex. 
\end{Remark}


\subsection{Interior gradient estimate}
We want to extend  Theorem \ref{Local Existence} for more general domains by using the Perron's method, and for this we will need to get a compactness theorem for solutions of \eqref{soliton}. This compactness result can be obtained from the following interior gradient estimate whose proof follows ideas that were used in \cite[Appendix A]{eichmair-metzger} and \cite[Theorem 12.1]{Hoffman}.

\begin{Proposition}[Interior gradient estimate]\label{Interior gradient estimate}
Let \(\{u_n\}\) be a sequence of solutions of \eqref{soliton} on a domain $\Omega\subset\mathbb{P}$. Let $p\in\Omega$ and $r>0$ be small enough so that the $g_0$-geodesic ball $B_{2r}(p) \subset\subset \Omega$.  Assume that $|u_n(q)|\leq K$ for all $n \in\n$ and $q\in B_{2r}(p)$. Then there exists a constant $c>0$ such that 
	\[
		\sup_{q\in B_{r}(p)}h_c(\nabla u_n(q),\nabla u_n(q))\leq c\ \operatorname{for\ all}\ n\in\n.
	\]
\end{Proposition}
\begin{proof}
Suppose the contrary. Then, up to extracting a subsequence, we find a sequence $\{x_n\} \subset B_r(p)$ such that 
	\[
	h_c(\nabla u_n(x_n),\nabla u_n(x_n)) \to \infty
	\] 
as $n\to\infty$. Since $\overline{B}_r(p)$ is compact in $(M\times\R,g_0)$ (see Remark \ref{Ilmanen's metric.}) we may assume that $x_n\to x_\infty$ in $(M\times\R,g_c)$.  On the other hand $\{u_n(x_n)\}$ is a bounded sequence, so we may also assume $u_n(x_n)\to\alpha$ as $x_n \to x_\infty$.

Now let $\Sigma_n \coloneqq \{ \Psi(x, u_n(x)) \colon x\in B_{2r}(p) \}$ be the Killing graph of $u_n$ over the ball $B_{2r}(p)$.  
Then $\{\Sigma_n\}$ is a sequence of stable $g_c$-minimal surfaces with locally bounded area in $\{\Psi(x,\R) \colon x\in B_{2r}(p)\}$ and therefore, by \cite[Theorem 3]{Schoen-Simon}(see also N. Wickramasekera \cite[Theorem 18.1]{Wickramasekera}), we may assume, up to a subsequence, that $\Sigma_n\to\Sigma_\infty$, where $\Sigma_\infty$ is a smooth stable minimal surface in the cylinder $\{\Psi(x,\R) \colon x\in B_{2r}(p)\}$. Note that $\Sigma_\infty$  is not empty because $\Psi(x_\infty,\alpha)\in\Sigma_\infty$.

\vspace{12pt}

\noindent
Now we claim that any connected component of $\Sigma_\infty$ is a smooth graph.
On the contrary, suppose that there exists a connected component $S\subset \Sigma_\infty$ that is not a graph over a subset of $B_{2r}(p).$ Because each $\Sigma_n$ is a graph over $B_{2r}(p)$, and $\Sigma_n \to \Sigma_\infty$ smoothly, we obtain that any horizontal line $\Psi(q,\R)$, $q\in B_{2r}(p)$, intersects $S$ in a connected subset. Since we are assuming that $S$ is not a graph, there exists a horizontal line $\Psi(q,\R)$, $q\in B_{2r}(p)$, such that $\Psi(q,[a,a+\epsilon])\subset S.$ 

Let $S(\theta) \coloneqq \{ \Psi ((x,t),s+\theta) \colon \Psi((x,t),s)\in S \}$ be a translation of $S$ by $\theta$ in the direction of $\partial_s$ along the flow $\Psi$. Because $\Psi(q,[a,a+\epsilon]) \subset S$, the maximum principle implies that $S(\theta) = S$ for all $\theta \in (0,\epsilon)$ and it follows that $S$ is a cylinder $\Psi(S',\R)$, where $S'$ is a curve in $B_{2r}(p)$. But this is a contradiction since each $\Sigma_n \subset \{ \Psi(x,[-K,K]) \colon x\in B_{2r}(p) \}$. Therefore $S$ is a Killing graph of a continuous function $u_\infty$.

To conclude that $S$ is a graph of a smooth function, we will use a Rad\'o-Alexandrov type argument. For this, we denote by
	\[
	\Lambda_\beta \eqqcolon \{ \Psi((x,t),\beta) \colon (x,t)\in\mathbb{P},\, \beta \in\R  \}
	\]
a foliation of $M\times\R$ by vertical planes. Define
	\[
	S_+(\beta) \coloneqq \{ \Psi((x,t),s) \in S \colon s\le \beta \} \,\text{ and } \,
	S_-(\beta) \coloneqq \{ \Psi((x,t),s) \in S \colon s\ge \beta \}
	\]
to be the parts of $S$ that lie on different sides of $\Lambda_\beta$, and
	\[
	S_+^* (\beta) \coloneqq \{ \Psi((x,t), \beta - s) \colon \Psi((x,t),s) \in S_+  \}
	\]
the reflection of $S_+$ with respect to $\Lambda_\beta$. Since $S$ is a graph of a continuous function, $S_+^*(\beta)$ and $S_-(\beta)$ can intersect only along the boundary lying on the plane $\Lambda_\beta$.

Now assume that there exists a point $q = \Psi((x,t),u_\infty(x,t)) \in S$ so that the normal to $S$ at $q$ is perpendicular to $\partial_s$. Then reflecting with respect to the plane $\Lambda_{u_\infty(x,t)}$ through $q$, we obtain that $S_+^*(u_\infty(x,t))$ and $S_-(u_\infty(x,t))$ are intersecting along the plane $\Lambda_{u_\infty(x,t)}$, and they have a common tangent plane at $q$ so that locally they lie on different sides of this tangent plane. Now the maximum principle implies that $S_+^*(u_\infty(x,t)) = S_-(u_\infty(x,t))$ but this is a contradiction since $S$ was a graph. Therefore $S$ is a graph of a smooth function.

Finally, let $S$ be the connected component of $\Sigma_\infty$ containing $\Psi(x_\infty,\alpha).$ Then from the assumption $h_c(\nabla u_n(x_n),\nabla u_n(x_n)) \to \infty$, $n\to\infty$, it follows that the normal to $S$ at $\Psi(x_\infty,\alpha)$ is perpendicular to $\partial_s$. But this is a contradiction with $S$ being a graph of a smooth function over an open subset of $B_{2r}(p)$. Therefore, there exists a constant $c$ so that
	\[
		\sup_{q\in B_{r}(p)}h_c(\nabla u_n(q),\nabla u_n(q))\leq c\ \text{for all}\ n\in\n.
	\]
\end{proof}

\begin{Remark}
Although we wrote the proof of the previous theorem for dimension 2, it works until dimension 6 due to the regularity of stable minimal hypersurfaces.
\end{Remark}

Once we have the interior gradient estimate, using Arzel\`{a}-Ascoli theorem and the theory of elliptic PDEs, we obtain the following compactness theorem.

\begin{Proposition}[Compactness Theorem]\label{Compactness Theorem}
Let \(\{u_n\}\) be a sequence of solutions of \eqref{soliton} on a domain $\Omega\subset\mathbb{P}.$ Suppose that $\{u_n\}$ is locally bounded on compact subsets of $\Omega.$ Then there exists a subsequence of $\{u_n\}$ that converges smoothly on compact subsets of $\Omega$ to a solution $u$ of \eqref{soliton}.
\end{Proposition}

\subsection{Perron's method}As we noted earlier, the statement of Theorem \ref{Local Existence} is only local, so to construct translating solitons over more general domains we need to use the Perron's method.  A detailed presentation of this method can be found, among others, in \cite[Section 2.8]{Gilbarg-Thudinger}.

Given $u\in C^0(\Omega)$, we say that $u$ is a {\it subsolution} in $\Omega \subset \mathbb{P}$ if for all $A\subset\subset\Omega$ and every solution $v$ of \eqref{soliton} such that $u\leq v$ on $\partial A,$ we have $u\leq v$ in $A$. A {\it supersolution} is defined similarly but with opposite inequalities. It is standard to check the following properties.

\begin{enumerate}[(i)]
\item A function $u\in C^2(\Omega)$ is a subsolution (supersolution) if and only if 
	\[
	\Div_{\mathbb{P}} \left(f^2\frac{\nabla u}{W}\right) \geq0 \ \left(\Div_{\mathbb{P}} \left(f^2\frac{\nabla u}{W}\right)\leq0\right);
	\]

\item Suppose that $\Omega$ is a bounded domain. Let $u\in C^0(\Omega)$ be a subsolution and $v\in C^0(\Omega)$ be a supersolution such that $u\leq v$ on $\partial\Omega,$ then $u\leq v$ in $\Omega;$

\item Let $u$ be a subsolution in $\Omega$ and $A$ be a subset strictly contained in $\Omega.$ Assume that $v\in C^2(A)$ is a solution of \eqref{soliton} with $v=u$ on $\partial A.$ Define a function $U\in C^0(\Omega)$ (called lifting of $u$ in A) given by
\[
	U(p)\coloneqq 
		\begin{cases}
		v(p), \quad p\in A \\
		u(p), \quad p \in\Omega\setminus A.
	\end{cases}
\]
Then $U$ is a subsolution in $\Omega$. Similar result holds also for supersolutions;

\item If $u_1,\ldots,u_r$ are subsolutions in $\Omega$, then $u:=\max\{u_1,\ldots,u_r\}$ is a subsolution in $\Omega;$

\item If $u_1,\ldots,u_r$ are supersolution in $\Omega$, then $u:=\min\{u_1,\ldots,u_r\}$ is a supersolution in $\Omega$.

\end{enumerate}

Suppose that $\Omega$ is a bounded domain and let $c$ be a bounded function on $\partial\Omega.$ We say that a function $u\in C^0(\Omega)$ is a subfunction (superfunction) relative to $c$ if $u$ is a subsolution (supersolution) in $\Omega$ and $u\leq c$ ($u\ge c$) on $\partial\Omega.$ Observe that by (iii) $\inf c$ is a subfunction relative to $c$ and $\sup c$  is a superfunction relative to $c$. Denote by $\mathcal{S}_c$ the set of all subfuctions relative to $c$. 

\begin{Theorem}[Perron's method]\label{Perron's method}
The function $u(p)\coloneqq \sup_{v\in\mathcal{S}_c} v(p)$ is a smooth solution of \eqref{soliton}. 
\end{Theorem}
\begin{proof}
The proof follows the same strategy as in \cite{Gilbarg-Thudinger,Nguyen}. Here we can take the lifting of any subsolution over sufficiently small $f$-geodesic balls contained in $\Omega$ since these balls are $f$-convex, i.e., the scalar $f$-curvature of the $f$-geodesic sphere is non-negative.
\end{proof}

Suppose that $\Omega$ is a bounded admissible domain with $\p\Omega=\bigcup J_i$, where $J_i$'s are arcs on $\p\Omega$ so that  $J_i\cap J_k$ is either an endpoint of both arcs or is empty, and let  $c=\{c_i \colon J_i\to\R\}$ be a family of bounded continuous functions. Then, as in \cite{Nguyen}, we get that the solutions of \eqref{soliton} given by Theorem \ref{Perron's method} have the desired boundary values. 

\begin{Theorem}[Perron's method-boundary data]\label{Perron's method-boundary data}
Suppose that $u$ is the solution given by Theorem \ref{Perron's method}. Then $u$ satisfies $u=c_i$ on $\operatorname{int}J_i$. 
\end{Theorem}

\subsection{Maximum principle} Next we prove the following version of the maximum principle.

\begin{Theorem}[Maximum principle]\label{Max. Principle}
Let $\Omega\subset \mathbb{P}$ be a bounded admissible domain. Suppose that $u_1$ and $u_2$ are solutions of \eqref{soliton} such that 
	\[
	\liminf_{x\to\partial\Omega}(u_2(x)-u_1(x))\geq0
	\]
with possible exception of finite number of points $\{q_1,\ldots,q_r\} \eqqcolon E \subset\partial\Omega.$ Then $u_2\geq u_1$ in $\Omega$ with strict inequality unless $u_2=u_1.$
\end{Theorem}
\begin{proof} The proof follows similar strategy as Spruck \cite[Section 1]{Spruck-Infinite}. We define a function $\varphi\colon\Omega\to\R$,
\[
\varphi =  
	\begin{cases}
		K-\epsilon, &\text{if }\ u_1-u_2\geq K; \\
		u_1-u_2 -\epsilon, &\text{if }\ \epsilon<u_1-u_2\leq K;  \\
		0, &\text{if }\ u_1-u_2\leq\epsilon,
	\end{cases}
\]
where $K,\epsilon>0$ are constants, $K$ large and $\epsilon$ small. Then $\varphi$ is piecewise smooth with bounded Lipschitz constant. Therefore, $\varphi$ is Lipschitz function with $0\leq\varphi\leq K$, $\nabla\varphi=\nabla u_1-\nabla u_2$ in the set $\{x\in \Omega \colon \epsilon < u_1(x) - u_2(x) < K\}$ and $\nabla\varphi=0$ almost everywhere in the complement of this set.

For each point $q_i\in E$,  let $B_\epsilon(q_i)$ be an open geodesic disk with center $q_i$ and radius $\epsilon$. Denote $\Omega_\epsilon = \Omega \setminus \cup_i B_\epsilon(q_i)$ and suppose that $\partial\Omega_\epsilon=\tau_\epsilon \cup \rho_\epsilon$, where $\rho_\epsilon = \cup_i (\partial B_\epsilon(q_i)\cap\Omega)$ and $\tau_\epsilon=\partial\Omega_{\epsilon}\cap\partial\Omega.$ Since \(\liminf(u_2-u_1)\geq0\) in $\partial\Omega\setminus E,$ we have $\varphi\equiv0$ in a neighbourhood of $\tau_\epsilon.$ Define 
\begin{equation}\label{Equation J}
J:=\int_{\rho_\epsilon}\varphi\left[ h_c\left(f^2\frac{\nabla u_1}{W_1},\nu\right)-h_c\left(f^2\frac{\nabla u_2}{W_2},\nu\right)\right],
\end{equation}
where $\nu$ is the unit outer normal to $\Omega_\epsilon$ and $W_i = \sqrt[]{1+f^2h_c(\nabla u_i,\nabla u_i)}.$ From \eqref{Equation J}, and $0\leq\varphi\leq K$, we obtain

\begin{equation}\label{Ineq. upper J}
J\leq 2K\sum_{i=1}^{r}{\rm L}_f[\partial B_\epsilon(q_i)].
\end{equation}
On the other hand, since $\varphi$ is a Lipschitz function, we have

\begin{align*}
\Div_{\mathbb{P}} \left[\varphi \left( f^2\frac{\nabla u_1}{W_1}-f^2\frac{\nabla u_2}{W_2} \right) \right] 
&=
 h_c\left(\nabla\varphi, f^2\frac{\nabla u_1}{W_1}-f^2\frac{\nabla u_2}{W_2}  \right) \\
&+ 
\varphi\left[ \Div_{\mathbb{P}} \left(f^2\frac{\nabla u_1}{W_1}\right)-\Div_{\mathbb{P}} \left(f^2\frac{\nabla u_2}{W_2}\right)\right],
\end{align*}
almost everywhere in $\Omega$. Therefore, by the divergence theorem, we obtain
\begin{align}\label{Ineq. Under J}
\nonumber J 
&= 
\int_{\Omega_\epsilon} \left[ h_c \left(\nabla\varphi, f^2\frac{\nabla u_1}{W_1}-f^2\frac{\nabla u_2}{W_2} \right)+\varphi\left(\Div_{\mathbb{P}} \left(f^2\frac{\nabla u_1}{W_1}\right)-\Div_{\mathbb{P}} \left(f^2\frac{\nabla u_2}{W_2}\right)\right) \right]  \\
&\geq 
\int_{\Omega_\epsilon} h_c\left(\nabla\varphi, f^2\frac{\nabla u_1}{W_1}-f^2\frac{\nabla u_2}{W_2} \right).
\end{align}
Now if $N_i = \frac{\partial_s}{fW_i}-f\frac{\nabla u_i}{W_i},$ then
\begin{align}\label{Ineq. normal}
\nonumber h_c\left(\nabla u_1-\nabla u_2, f^2\frac{\nabla u_1}{W_1}-f^2\frac{\nabla u_2}{W_2} \right) &= g_c\left( N_1-N_2,W_1N_1- W_2N_2\right) \\
\nonumber &= W_1-(W_1+W_2)g_c(N_1,N_2)+W_2 \\
&= \frac{1}{2}(W_1+W_2)g_c(N_1-N_2,N_1-N_2).
\end{align}

From \eqref{Ineq. upper J}, \eqref{Ineq. Under J} and \eqref{Ineq. normal} we get
\[
 2K\sum_{i=1}^{r} {\rm L}_f[\partial B_\epsilon(q_i)] \geq\frac{1}{2}\int_{\Omega_\epsilon\cap\{0<u_1-u_2<K\}}(W_1+W_2)g_c(N_1-N_2,N_1-N_2) \geq 0,
\]
and in particular, letting $\epsilon\to 0$ we arrive at
\[
\int_{\{0<u_1-u_2<K\}}(W_1+W_2)g_c(N_1-N_2,N_1-N_2)=0.
\]
Therefore $N_1=N_2$ in $\{x\in\Omega \colon 0<u_1-u_2<K\}$, and consequently also $\nabla u_1=\nabla u_2$ in the same set. As $K$ was arbitrary, we conclude that $\nabla u_1=\nabla u_2$ whenever $u_1 > u_2$. 

To finish the proof, assume that $\{0<u_1-u_2\}$ contains a connected component with non-empty interior. Then, by the previous argument, $u_1=u_2+c,$ where $c$ is a positive constant, and consequently by \cite[Theorem 10.1]{Gilbarg-Thudinger} we have $u_1=u_2+c$ in $\Omega.$ On the other hand, as \(\liminf(u_2-u_1)\geq0\) for any approach of $\partial\Omega\setminus E$, $c$ must be a non-positive constant, which is impossible, and therefore $u_2\ge u_1$.
\end{proof}

\subsection{Scherk's translator barrier}
Next we will construct a specific solution that looks like a part of the Scherk's surface. This will be used later as a barrier to get information about sequences of solutions of \eqref{soliton}. The proof of the following proposition follows similar ideas as in \cite{Nelli-Rosenberg}, \cite{Pinheiro} and \cite{Nguyen} but there are some differences so we will write it completely.

\begin{Proposition}[Scherk's surface]\label{Scherk}
Let $\Omega\subset\mathbb{P}$ be a geodesically $f$-convex and admissible domain whose boundary $\partial\Omega$ is a union of four $f$-geodesic arcs $A_1, A_2, C_1$ and  $C_2$ so that $A_1$ and $A_2$ do not have common endpoints. Assume also that 
\[
{\rm L}_f[A_1]+{\rm L}_f[A_2]<{\rm L}_f[C_1]+{\rm L}_f[C_2].
\]
Then, given any bounded continuous data $c_i \colon C_i\to\R$, there exists a solution $u$ of \eqref{soliton} such that $u = c_i$ on $C_i$ and $u\to\infty$ along $A_1\cup A_2.$
\end{Proposition}
\begin{proof} We divide the proof into two cases depending on the continuous boundary data $c_i$.

\textbf{Case $c_1=c_2\equiv0$.} Consider the sequence of curves \(\{\gamma_n\}\subset \Psi(\p\Omega,\R)\), where $\gamma_n(x)=\Psi(x,0)$ for all $x\in \operatorname{C}_1\cup \operatorname{C}_2,$ $\gamma_n(x)=\Psi(x,n)$ for all $x\in A_1\cup A_2$ and $\gamma_n$ is a horizontal segment joining the vertices  $\Psi(x,0)$ and $\Psi(x,n)$ when $x$ is a vertex of $\partial\Omega.$ By Theorem \ref{Local Existence} (or Theorem \ref{Perron's method}) there exists a solution $u_n\colon \Omega\to\R$ of \eqref{soliton} with the continuous curve $\gamma_n$ as boundary. Moreover, by Theorem \ref{Max. Principle} the sequence $\{u_n\}$ is monotone increasing. So we need to prove that $\{u_n\}$ is locally bounded on compact subsets of $\Omega$ and hence, by Theorem \ref{Compactness Theorem}, we can obtain  a subsequence of $\{u_n\}$ converging smoothly on compact subsets of $\Omega$ to a solution $u$ of \eqref{soliton} satisfying the required properties. 

In order to control the sequence on compact subsets of $\Omega$, we construct a minimal annulus, and for this, consider the minimal disk $D_i^h \coloneqq C_i\times[0,h]=\Psi(C_i,[0,h]),$ that is the Killing rectangle over $C_i$ with height $h$. Then $D_i^h$ is area-minimizing with respect to the area functional. Indeed, suppose that $\Sigma$ is any minimal disk with boundary $\partial D_i^h$. 

Recall that we are considering the metric $g_c$ in $M\times\R$, and so we equip $\Sigma$ with the Riemannian metric that is the restrictions of $g_c$ onto $\Sigma$. If we write $h_\Sigma = s|_\Sigma$ as the ``height function" of $\Sigma$, we see that
\begin{equation}\label{gradient}
\nabla h_\Sigma=  (\bar\nabla s)^\top =\frac{\partial_s}{f^2}-g_c\left(N_\Sigma,\frac{\partial_s}{f^2}\right)N_\Sigma.
\end{equation}
Taking the divergence we can conclude that
\begin{align*}
\nonumber \Delta^\Sigma h_\Sigma 
&= \Div_\Sigma(\nabla h_\Sigma)=\Div_\Sigma\left(\frac{\partial_s}{f^2}-g_c\left(N_\Sigma,\frac{\partial_s}{f^2}\right)N_\Sigma\right) \\
\nonumber 
&= \Div_\Sigma\left(\frac{\partial_s}{f^2}\right)-g_c\left(N_\Sigma,\frac{\partial_s}{f^2}\right)\Div_\Sigma\left(N_\Sigma\right)=g_c\left(\p_s,\nabla_\Sigma\left(\frac{1}{f^2}\right)\right) \\
\nonumber 
&= -2g_c\left(\nabla^\Sigma \log f, \nabla^\Sigma h_\Sigma\right),
\end{align*}
i.e., $\Delta^\Sigma h_\Sigma + 2g_c\left(\nabla^\Sigma \log f, \nabla^\Sigma h_\Sigma\right)=0$, and hence $h_\Sigma$ is harmonic with respect to the weighted Laplacian. Now, by the maximum principle, the maximum and the minimum of $h_\Sigma$ are attained at the boundary of $\Sigma.$ 
Therefore the co-area formula gives
\[{\rm Area}[\Sigma]=\int_\Sigma d\mu_\Sigma=\int^{h}_{0}\int_{h_{\Sigma}^{-1}(t)}\frac{1}{\sqrt[]{g_c(\nabla h_{\Sigma},\nabla h_{\Sigma})}}ds_tdt,
\]
where $d\mu_\Sigma$ denotes the Riemannian measure induced by $g_c$ in $\Sigma.$ From \eqref{gradient} one obtains $\frac{1}{f^2}\geq g_c(\nabla h_{\Sigma},\nabla h_{\Sigma}).$  Consequently by Proposition \ref{minimize charact.} and Remark \ref{f-prop.} \eqref{f-Area} we have

\begin{align*}
{\rm Area}[\Sigma]&\geq  \int^{h}_{0}\int_{h_{\Sigma}^{-1}(t)}fds_tdt=\int^{h}_{0}{\rm L}_f[h_{\Sigma}^{-1}(t)]dt \\
&\geq \int^{h}_{0}{\rm L}_f[C_i]dt={\rm Area}[C_i\times[0,h]]={\rm Area}[D_i^h]
\end{align*}
and $D_i^h$ is area-minimizing with respect to the area functional. 

Now, to construct the annulus, consider first the piecewise annulus 
	\[
	\mathcal{C}_h:=\Omega\cup\Omega_h\cup (A_1\times[0,h])\cup (A_2\times[0,h]),
	\] 
where $\Omega_h \coloneqq \{\Psi((x,t),h)\colon (x,t)\in\Omega\}$. As 
	\[
	{\rm Area}(\mathcal{C}_h)=2{\rm Area}(\Omega)+{\rm Area}[A_1\times[0,h]]+{\rm Area}[A_2\times[0,h]],
	\] 
it holds
\begin{align*}
{\rm Area}[\mathcal{C}_h] &- {\rm Area}[D^h_1]-{\rm Area}[D^h_2] \\
&= 2{\rm Area}[\Omega]+h({\rm L}_f[A_1]+{\rm L}_f[A_2]-{\rm L}_f[C_1]-{\rm L}_f[C_2]) \\
&<0,
\end{align*}
if $h\geq h_0$ for some $h_0$ large enough. Consequently by \cite[Theorem 3.1]{Morrey} or \cite[Theorem 1]{Meeks-Yau-1} (see also \cite{Meeks-Yau-2}) for each $h\geq h_0$ there is a minimal annulus $\Theta_h$ with boundary $\partial D_1^h$ and $\partial D_2^h.$

Finally, to conclude the proof, we need to observe that essentially the same argument of Nelli and Rosenberg \cite[Theorem 2]{Nelli-Rosenberg} and Pinheiro \cite[Theorem 1.4]{Pinheiro} allows us to prove that the family of annuli $\{\Theta_h\}_{h\geq h_0}$ is an upper barrier for $\{u_n\}$. This means that all horizontal lines over points of $\Omega$ intersect ${\rm Graph}[u_n]$ before intersecting $\Theta_h$, i.e., $\Theta_h$ is above ${\rm Graph}[u_n]$ for all $n$ and all $h\geq h_0$. Moreover, the horizontal projection of $\{\Theta_h\}_{h\geq h_0}$ over $\Omega\cup C_1\cup C_2$ is an exhaustion of $\Omega\cup C_1\cup C_2.$ Therefore $\{u_n\}$ is a locally bounded sequence on compact subsets of $\Omega\cup C_1\cup C_2,$ and this finishes the proof of this case.

\textbf{General case} (c is a bounded function). Suppose that $|c_i|\leq K$ and let $v\colon \Omega\cup C_1\cup C_2\to\R$ be the function of the first case. Let  $\{\gamma_n \subset \Psi(\p\Omega, \R)\}$ be the sequence of curves, where $\gamma_n(x)=\Psi(x,\min\{n,c_i(x)\})$ for all $x\in \operatorname{C}_i,$ $\gamma_n(x)=\Psi(x,n)$ for all $x\in A_1\cup A_2$ and $\gamma_n$ is a horizontal segment joining the vertices  $\Psi(x,0)$ and $\Psi(x,n)$ when $x$ is a vertex of $\partial\Omega.$ By Theorem \ref{Perron's method} there exists a solution $u_n \colon \Omega\to\R$ of \eqref{soliton} with continuous boundary curve $\gamma_n.$ Moreover, by Theorem \ref{Max. Principle} the sequence $\{u_n\}$ is monotone non-decreasing and $-K\leq u_n\leq v+K$ in $\Omega$. Hence by Theorem \ref{Compactness Theorem} we obtain that $\{u_n\}$ converges smoothly on compact subsets of $\Omega$ to a solution $u$ of \eqref{soliton} with the required properties.
\end{proof}

\begin{remark}
To control the boundary value of the limit graph $u$ we use barriers that ultimately yield boundary height and gradient estimates for minimal surface equation in Riemannian ambients. To a clear presentation of this procedure we refer the reader to  \cite[Theorem 6]{Klaser-Menezes} and references therein.
\end{remark}

As an application of this result we can prove the following.
\begin{Proposition}\label{barrier}
Let $\Omega \subset \mathbb{P}$ be a bounded domain such that $\p\Omega$ is a union of an $f$-geodesic arc $A$ and an $f$-convex arc $C$ with their endpoints. Assume there exists a geodesically $f$-convex domain $\Omega'\subset\mathbb{P}$ so that $\Omega\subset\Omega'$ and its boundary $\partial\Omega'$ is a union of four $f$-geodesic arcs $A_1, A_2, C_1$ and  $C_2$ so that $A_1$ and $A_2$ do not have common endpoints and $A\subset A_1$. Moreover assume that 
	\[
	{\rm L}_f[A_1]+{\rm L}_f[A_2]<{\rm L}_f[C_1]+{\rm L}_f[C_2].
	\]
Then, given any bounded continuous function $c \colon C\to\R$, there exists a solution of \eqref{soliton} in $\Omega$ such that $u\to\infty$ on $A$ and has the continuous boundary data $c$ on $C.$
\end{Proposition}
\begin{proof}
Let $\{\gamma_n\} \subset \Psi(\p\Omega,\R)$ be a sequence of curves, where 
\[
\gamma_n(x)=\Psi(x,\min\{c(x),n\})
\]
for all $x\in C$, $\gamma_n(x)=\Psi(x,n)$ for all $x\in A$ and $\gamma_n$ is a horizontal segment joining the vertices  $\Psi(x,0)$ and $\Psi(x,n)$ when $x$ is a vertex of $\partial\Omega.$ By Theorem \ref{Perron's method} there exists a solution $u_n \colon \Omega\to\R$ of \eqref{soliton} with continuous boundary curve $\gamma_n.$ Moreover, by Theorem \ref{Max. Principle} the sequence $\{u_n\}$ is monotone and increasing. Now, if $v$ denotes the function over $\Omega'$ given by the previous result, then we must have 
	\[
	\inf_{C}c\leq u_n\leq \sup_{C}c+v\ \operatorname{in}\ \Omega
	\] by Theorem \ref{Max. Principle}. 
In particular, by Theorem \ref{Compactness Theorem}, $u_n$ converges on compact subsets of $\Omega$ to a solution $u$ of \eqref{soliton} with the required properties.
\end{proof}

\begin{Proposition}\label{Continuous data}
Let $\Omega \subset \mathbb{P}$ be a domain. Suppose that $\gamma$ is an $f$-convex arc in $\p\Omega$. Let $\{u_n\}$ be a sequence of solutions of \eqref{soliton} that converges uniformly to $u$ on compact subsets of $\Omega$. Suppose that $u_n \in C^0(\Omega\cup\gamma)$ and $u_{n}|_\gamma$ converge uniformly on compact subsets of $\gamma$ to a function $c \colon \gamma\to\R$ that is continuous or $c\equiv\pm\infty$. Then $u$ is continuous in $\Omega\cup\gamma$ and $u|_{\gamma}=c.$
\end{Proposition}
\begin{proof}
Given $p\in\gamma$, assume that $c(p)>M$, where $M$ is a fixed constant. Clearly, we only need to prove that there exists a neighbourhood $U$ of $p$ in $\Omega\cup\gamma$ so that $u>M$ in $U$ to conclude that $u$ is continuous at $p$. The same argument works if we want to prove the existence of this neighbourhood when $c(p)<M$. 

Fix a constant $N\in(M,c(p))$. Since $u_n|\gamma$ converge uniformly to $c$ on compact subsets of $\gamma$, there exists a subarc $\lambda\subset\gamma$ containing $p$ in its interior so that $u_n>N$ for all $n\geq n_0$ on $\lambda,$ for some $n_0$ large enough. Moreover, we can assume that $\lambda$ lies in a neighbourhood of $p$ which is geodesically $f$-convex by taking $\lambda$ small enough. Notice also that, if $\lambda$ is small enough, we have two cases to analyse:
		\begin{enumerate}
			\item $\lambda$ is an $f$-geodesic;
			\item $\lambda$ is strictly $f$-convex, except possibly at $p$.
		\end{enumerate}
Suppose $\lambda$ is an $f$-geodesic, then we can construct an admissible domain $\Delta\subset\Omega$ with four edges $A_1,A_2,\lambda'$ and $\lambda$ so that $A_1$ and $A_2$ do not have common endpoints, and ${\rm L}_f[A_1]+{\rm L}_f[A_2]<{\rm L}_f[\lambda']+{\rm L}_f[\lambda]$. By Proposition \ref{Scherk} there exists a solution $v$ of \eqref{soliton} so that $v\to - \infty$ along $A_1\cup A_2$, $v=N$ on $\lambda$ and $v=M'$ on $\lambda'$, where $M'=\inf_{\lambda'}u_n>-\infty$, since $u_n$ converge on compact subsets to $u.$ Now by Theorem \ref{Max. Principle} we conclude $v<u$ in $\Delta$.\\

On the other hand, if $\lambda$ is strictly $f$-convex, except possibly at $p$, then by our hypothesis on $\lambda$ there exists an $f$-geodesic $\eta$ joining the endpoints of $\lambda.$ By Proposition \ref{barrier} there exists a solution $v$ of \eqref{soliton} so that $v\to - \infty$ along $\eta$ and $v=N$ on $\lambda.$ Again by Theorem \ref{Max. Principle} we must have  $v<u$ in $\Delta$, where $\Delta$ is a domain in $\Omega$ with boundary $\lambda\cup\eta.$ In particular, in both cases there exists a small neighbourhood $U$ of $p$ in $\Omega\cup\gamma$ so that $u\geq N>M$ in $U.$ 
\end{proof}

Notice that if we break a curve $\gamma$ into small parts and use Proposition \ref{barrier}, we can conclude the following.
\begin{Lemma}\label{Local Control}
Let $\Omega\subset \mathbb{P}$ be a bounded domain and $\gamma\subset\p\Omega$ be a strictly $f$-convex curve with respect to inner unit normal to $\partial\Omega$. Suppose that $\{u_n\}$ is a sequence of solutions of \eqref{soliton} in $\Omega$ such that $u_n\geq c$ $($respectively $u_n\leq c)$ on $\gamma$, where $c$ is a constant. Then given any compact subarc $\lambda \subset \gamma$ there exists a neighbourhood $U(\lambda)$ of $\lambda$ in $\overline{\Omega}$ and a constant $K(\lambda)>0$ such that $u_n\geq c-K(\lambda)$  $($respectively $u_n\leq c+K(\lambda))$ for all n in $U(\lambda)$.
\end{Lemma}

\subsection{Straight line lemma}
In this subsection we give a geometric proof of the classical straight line lemma. The ideas that we will develop are inspired by work of Eichmair and Metzger in \cite[Appendix B]{eichmair-metzger} (see also \cite[Theorem 12]{GHLM}). 
\begin{Lemma}[Straight line lemma]\label{Straight line}
Let $\Omega \subset \mathbb{P}$ be a domain such that $\gamma \subset \p\Omega$ is a smooth open arc and suppose that $u\colon\Omega \to \R$ is a solution of \eqref{soliton}. If $u (x)\to\pm\infty$ when $x\to\gamma$, then $\gamma$ is an $f$-geodesic.
\end{Lemma}
\begin{proof}
Fix $x\in\gamma$ and take a sequence $\{x_i\}\subset \Omega$ with $x_i\to x.$ Assume that $u(x_i)\to\infty$ and define the sequence of hypersurfaces $\{\Sigma_i={\rm Graph}[u-u(x_i)]\}$.

Fix a closed geodesic ball $B_\epsilon[(x,0)]$ around $(x,0)( \in \gamma\times\R)$ in $M\times\R$ so that $B_\epsilon[(x,0)]$ does not intersect the lines over the endpoints of $\gamma.$ Notice that each $\Sigma_i$ intersects $B_\epsilon[(x,0)]$ for $i$ sufficiently large, and therefore we can suppose that each $\Sigma_i$ intersects $B_\epsilon[(x,0)]$. For each $i$, let $S_i$ be the connected component of $\Sigma_i\cap B_\epsilon[(x,0)]$ so that $(x_i,0)\in S_i.$ Since
\begin{equation*}
{\rm Area}[S_i]\leq{\rm Area}[\partial B_\epsilon[(x,0)]]
\end{equation*}
and $S_i$ is stable for all $i$, then \cite[Theorem 3]{Schoen-Simon} implies that, up to extracting a subsequence, we may assume $S_i\to S_{\infty}$ in $\operatorname{int}B_\epsilon((x,0)).$ Since $(x,0)\in S_{\infty}\subset\gamma\times\R$, we obtain that ${\rm k}_f[\gamma](x)={\rm H}_{\gamma\times\R}(x,0)=0.$ Consequently we must have ${\rm k}_f[\gamma]\equiv 0$, so $\gamma$ is an $f$-geodesic.
\end{proof}
\begin{Remark}
The proof of Lemma \ref{Straight line} works until dimension at most 7 and the hypothesis of smoothness of $\gamma$ can be omitted, see \cite[Appendix B]{eichmair-metzger} for more details. 
\end{Remark}
\begin{Remark}\label{estimate of normal}
We point out that the proof of Lemma \ref{Straight line} gives actually a stronger property: If $u \colon \Omega \to \R$ is a solution of \eqref{soliton} and $\gamma\subset\p\Omega$ is an $f$-geodesic, then for every $\delta \in(0,1)$ and every compact arc $\lambda \subset \gamma$ there exists $\eta(\delta,\lambda)>0$ so that if $\dist(p,\lambda)<\eta$, then
\begin{equation*}
1\geq h_c\left(f\frac{\nabla u}{W},\nu\right)(p)\geq 1-\delta\ ,\ \text{if}\ u\to+\infty\ \text{along } \lambda
\end{equation*}
and
\begin{equation*}
-1\leq h_c\left(f\frac{\nabla u}{W},\nu\right)(p)\leq -1+\delta\ ,\ \text{if}\ u\to-\infty\ \text{along } \lambda. 
\end{equation*}
\end{Remark}

\subsection{Flux formula} Let $\Omega\subset\mathbb{P}$ be a domain such that $\partial\Omega$ is piecewise $C^1$ smooth. From \eqref{soliton}, with the divergence theorem, we conclude that
	\[
	\int_{\p\Omega} \frac{f^2}{W} h_c(\nabla u, \nu) = 0,
	\]
where $\nu$ is unit outer normal to $\p\Omega$. This motivates us to define a flux
	\begin{equation}\label{def-flux}
	{\rm F}_u[\gamma] = \int_\gamma \frac{f^2}{W} h_c(\nabla u, \nu).
	\end{equation}
Next we collect some standard properties of the flux formula. The proofs of the next lemmas can be found in \cite{Hauswirth-Rosenberg-Spruck,Jenkins-Serrin,Nelli-Rosenberg,Nguyen,Pinheiro,Spruck-Infinite,Younes}.

\begin{Lemma}\label{Flux-1}
Let $u$ be a solution of \eqref{soliton} in an admissible domain $\Omega$. Then one has
\begin{enumerate}[(i)]
\item For all piecewise smooth polygon $\mathcal{P}$(not necessary admissible) in $\Omega$ we have 
	\[
	{\rm F}_u[\p\mathcal{P}]=0;
	\]

\item for every curve $\gamma$ in $\Omega$ we have
	\[
	|{\rm F}_u[\gamma]| \le {\rm L}_f[\gamma];
	\]
\item if $\gamma \subset \p\Omega$ is an $f$-geodesic such that $u$ tends to $+\infty$ on $\gamma$, we have
	\[
	{\rm F}_u[\gamma] = {\rm L}_f[\gamma];
	\]

\item if $\gamma \subset \p\Omega$ is an $f$-geodesic such that $u$ tends to $-\infty$, we have
	\[
	{\rm F}_u[\gamma] = - {\rm L}_f[\gamma];
	\]
	
\item if $\gamma\subset \p\Omega$ is an $f$-convex curve , i.e., ${\rm k}_f[\gamma]\geq0$ along $\gamma$, such that $u$ is continuous and finite on $\gamma$, then
	\[
	\left|{\rm F}_u[\gamma]\right| < {\rm L}_f[\gamma].
	\]	
\end{enumerate}
\end{Lemma}

\begin{Lemma}\label{Flux-2}
Let $\{u_n\}$ be a sequence of solutions of \eqref{soliton} on a domain $\Omega\subset\mathbb{P}$ so that $u_n$'s are continuous up to $\p\Omega.$ Consider an $f$-geodesic $\gamma\subset\p\Omega.$ Then
\begin{enumerate}[(i)]
\item if $\{u_n\}$ diverges uniformly to $+\infty$ on compact subsets of $\gamma$ while remaining uniformly bounded on compact subsets of $\Omega$, 
\[
\lim_{n\to\infty}{\rm F}_{u_n}[\gamma]={\rm L}_f[\gamma];
\]
\item if $\{u_n\}$ diverges uniformly to $-\infty$ on compact subsets of $\Omega$ while remaining uniformly bounded on compact subsets of $\gamma$, 
\[\lim_{n\to\infty}{\rm F}_{u_n}[\gamma]=-{\rm L}_f[\gamma].\]
\end{enumerate}
\end{Lemma}

\subsection{Divergence and convergence sets}The next step is to know the structure of the divergence and convergence sets of a sequence of solutions of \eqref{soliton}. 

\begin{Proposition}[Convergence set]\label{convergence}
Let $\{u_n\}$ be an increasing $($respectively decreasing $)$ sequence of solutions of \eqref{soliton} over a domain $\Omega\subset\mathbb{P}.$ Then there exists an open set $\mathcal{C}\subset\Omega$, called the convergence set, such that $\{u_n\}$ converges on compact subsets of $\mathcal{C}$ to a solution of \eqref{soliton} and diverges uniformly to $+\infty$ $($respectively $-\infty)$ on compact subsets of $\mathcal{D} \coloneqq \Omega\setminus\mathcal{C}$. The set $\mathcal{D}$ will be called the divergence set of $\{u_n\}.$ Moreover, if $\{u_n\}$ is bounded at a point $p\in\Omega$, then the convergence set $\mathcal{C}$ is non-empty.
\end{Proposition}
\begin{proof}
Suppose that $\{u_n\}$ is an increasing sequence. Given any point $p\in\mathcal{C}$ suppose, up to a subsequence, that $u_n(p)\to\alpha\in\R.$ Take $\epsilon$ small enough so that $\partial B_\epsilon(p)$ is a strictly $f$-convex curve, i.e., ${\rm k}_f[\partial B_\epsilon(p)]>0$, where $B_\epsilon(p)$ denotes the geodesic ball with center $p$ and radius $\epsilon$ in \(\mathbb{P}\). Consider the sequence of surfaces $\{\Sigma_n \coloneqq {\rm Graph}[u_n|_{B_\epsilon(p)}]\}$ in the solid Killing cylinder $B_\epsilon(p)\times\R.$ Since $\{\Sigma_n\}$ is a sequence of stable minimal surfaces with locally bounded area, then by \cite[Theorem 3]{Schoen-Simon}, up to a subsequence, we can suppose that $\{\Sigma_n\}$ converges smoothly to $\Sigma_\infty$ in $B_\epsilon(p)\times\R.$

As $u_1\leq u_n$ for all $n$, using an argument like in Proposition \ref{Interior gradient estimate}, one obtains that each connected component of $\Sigma_\infty$ is a smooth graph over an open subset of $B_\epsilon(p).$ On the other hand, as $\Psi(p,\alpha)\in\Sigma_\infty$ we can suppose that $\Psi(p,\alpha)\in\Sigma'\subset\Sigma_\infty$. Hence we can take $\delta$ small enough so that $u_n|_{B_{\delta}(p)}$ converges on compact subset to $u_\infty,$ where ${\rm Graph}[u_\infty|_{B_\delta(p)}]=\Sigma'\cap(B_\delta(p)\times\R).$ Therefore $B_\delta(p)\subset\mathcal{C}$ and this completes the proof that $\mathcal{C}$ is open and non-empty if there exists a point $p\in\Omega$ such that $\{u_n(p)\}$ is a bounded sequence. 

It is not hard to see that essentially the same proof works if we suppose that $\{u_n\}$ is a decreasing sequence.
\end{proof}

We can obtain the following properties of the structure of the divergence set.
Although the proof of these facts are rather standard, see for example \cite{Hauswirth-Rosenberg-Spruck,Jenkins-Serrin,Nelli-Rosenberg,Nguyen,Pinheiro,Spruck-Infinite,Younes}, we will prove them here for completeness.

\begin{Proposition}[Structure of the divergence set]\label{divergence}
Let $\Omega\subset\mathbb{P}$ be an admissible domain whose boundary is a union of $f$-convex arcs $C_i$. Let $\{u_n\}$ be either an increasing or a decreasing sequence of solutions of \eqref{soliton} over $\Omega$ such that for all open arcs $C_i$ the functions $u_n$ extend continuously to $C_i$ and either $u_n|_{C_i}$ converge uniformly to a continuous function or $+\infty$ or $-\infty,$ respectively. 
If $\mathcal{D}$ denotes the divergence set of $\{u_n\}$, then $\mathcal{D}$ satisfies the following properties.
\begin{enumerate}[(i)]
\item  $\p\mathcal{D}$ consists of a union of non-intersecting interior $f$-geodesics in $\Omega$, joining two points of $\p\Omega$, and arcs on $\p\Omega$. These arcs will be called chords. Moreover, a component of $\mathcal{D}$ cannot be an isolated point.

\item No two interior chords in $\p\mathcal{D}$ can have a common endpoint at a convex corner of $\mathcal{D}$.

\item A component of $\mathcal{D}$ cannot be an interior chord.

\item The endpoints of interior $f$-geodesic chords are among the vertices of $\p\Omega.$ 
\end{enumerate}
\end{Proposition}
\begin{proof}
Let us assume that $\{u_n\}$ is an increasing sequence. If $\mathcal{D}=\Omega$ there is nothing to prove, so we can suppose that $\mathcal{D}\neq\Omega$. Under this hypothesis, Lemma \ref{Straight line} implies that $\p\mathcal{D}$ consists of interior $f$-geodesics in $\Omega$ and arcs of $\p\Omega.$ We will prove initially that $\mathcal{D}$ cannot have isolated points. Indeed, if $p$ is an isolated point of $\mathcal{D}$, then we can construct a quadrilateral domain $\Omega'\subset\Omega$ satisfying the condition of Proposition \ref{Scherk} so that $p\in\operatorname{int}\Omega'$. Moreover, we can suppose that $\overline{\Omega'}$ does not intersect $\mathcal{D}\setminus\{p\}.$ Now consider $M=\sup_{C_1\cup C_2} |u_n|<\infty,$ where $C_1$ and $C_2$ denotes the edges of $\Omega'$ with greater total $f$-length. If $v$ denotes the function given by Proposition \ref{Scherk}, then by Theorem \ref{Max. Principle} one gets $-M-v\leq u_n\leq M+v$ in $\Omega'$ which is impossible since $u_n(p)\to+\infty.$ This contradiction shows that $\mathcal{D}$ cannot have isolated points. Note that this argument proves also that a chord of $\p\mathcal{D}$ cannot have an endpoint in the interior of $\Omega$.

Next we prove that the interior $f$-geodesics are non-intersecting. In fact, if the contrary of this was true, then we can construct a triangle $\Delta$ with edges $a_1,a_2$ and $a_3$ so that $a_1,a_2\subset\p\mathcal{D}$ and $\Delta$ lies either in $\mathcal{C}$ or in $\mathcal{D}.$  Assume first that $\Delta$ lies in $\mathcal{C}$. Then by Lemma \ref{Flux-1} (i) we have 
\begin{equation}\label{Triangle}
	0={\rm F}_{u_n}[\p\Delta]={\rm F}_{u_n}[a_1]+{\rm F}_{u_n}[a_2]+{\rm F}_{u_n}[a_3].
\end{equation}
Since $a_1$ and $a_2$ lie on $\p\mathcal{D}$ we have $\lim_n {\rm F}_{u_n}[a_i] = {\rm L}_f[a_i]$ for $i=\{1,2\}$, by Lemma \ref{Flux-2}. On the other hand, again by Lemma \ref{Flux-1}, we have $|{\rm F}_{u_n}[a_3]|\leq {\rm L}_f[a_3]$, so we get a contradiction with \eqref{Triangle}. Therefore we must have $\Delta\subset\mathcal{D}$. 
But now note that
$a_i\subset\p\mathcal{C}$ for $i=\{1,2\}$, and therefore by Lemma \ref{Flux-2}, we must have $\lim_n{\rm F}_{u_n}[a_i] = -{\rm L}_f[a_i]$ for $i=\{1,2\}$. Using the previous argument we arrive again to a contradiction, and this proves (i).

In order to get (ii), assume that there exist two interior chords $\gamma_1$ and $\gamma_2$ with a common endpoint $p\in\p\Omega.$ Again, we can construct a triangle $\Delta$ with edges $a_1,a_2$ and $a_3$ so that $a_1,a_2\subset\p\mathcal{D}$ and $\Delta$ lies either in $\mathcal{C}$ or in $\mathcal{D}.$ Then the same argument as above proves (ii).

To prove assertion (iii), suppose that $\gamma$ is an interior chord that is a connected component of $\mathcal{D}$. Fix any point $p\in\gamma$ which lies in $\operatorname{int}\Omega.$ Clearly we can construct a quadrilateral domain $\Omega'$ such that it satisfies the properties of Proposition \ref{Scherk}. If $\p\Omega'=A_1\cup A_2\cup C_1\cup C_2$, then $\gamma$ only intersects $A_1$ and $A_2$ at a unique interior point on these arcs and $\overline{\Omega'}$ does not intersect $\mathcal{D}\setminus\gamma$. Consider $M=\sup_{C_1\cup C_2}|u_n|<\infty$ and let $v \colon \Omega' \to \R$ be the function given by Proposition \ref{Scherk}. Using Theorem \ref{Max. Principle} one obtains $-M-v\leq u_n\leq M+v$ in $\Omega'$ which is impossible since an arc of $\gamma$ lies in $\Omega'.$ This concludes the proof of (iii).

Finally, assume that there exists a chord $\gamma$ with endpoint $p\in\operatorname{int}C_i$ for some $C_i \subset \partial\Omega$. If ${\rm k}_f[C_i](p)>0$ then Lemma \ref{Local Control} gives us a contradiction. On the other hand, if ${\rm k}_f[C_i](p)=0$, then we have two cases to check: either there is a sequence \(\{p_n\}\subset C_i\) so that $p_n\to p$ and ${\rm k}_f[C_i](p_n)>0$ or there is a subarc $\eta \subset C_i$ so that ${\rm k}_f[\eta]\equiv0$ and $p$ lies in the interior of $\eta.$ The first case would make it  possible to find a domain $\Delta$ satisfying the condition of Proposition \ref{barrier} so that $p$ lies in the interior of an arc of $\p\Delta$ which is not an $f$-geodesic and $\Delta\subset\overline{\Omega}$. Suppose first that $\{u_n\}$ is unbounded on $C_i$ and let $v_M \colon \Delta\to\R$ be the function given by Proposition \ref{barrier} with continuous data $M,$ where $M$ is a fixed constant. By Theorem \ref{Max. Principle} one has
\[-v_M<u_n\ \operatorname{in}\ \Delta\ \operatorname{for\ all}\ n \operatorname{large\ enough}.\]
Since $M$ was arbitrary, this implies that a small neighbourhood of $p$ lies in $\mathcal{D}$, but this is impossible because $\gamma\subset\p\mathcal{D}$. On the other hand, if $\{u_n\}$ is bounded on $C_i$ and $v \colon \Delta\to\R$ is the function given by Proposition \ref{barrier} with continuous data $M,$ where $M=\sup_C |u_n|,$ then by Theorem \ref{Max. Principle} one obtains $u_n\leq v$ in $\Delta$, which again leads to a contradiction. 

Hence, there exists a subarc $\eta$ of $C_i$ so that ${\rm k}_f[\eta] \equiv0$ and $p$ lies in the interior of $\eta.$ Again we have two cases to check: either $\{u_n\}$ is unbounded or $\{u_n\}$ is bounded on $C_i.$ If $\{u_n\}$ is unbounded on $C_i,$ then we can find a triangle $\Delta$ with edges $a_1,a_2$ and $a_3$ so that $a_1\subset\gamma$, $a_2\subset C_i$ and $a_3$ lies in $\mathcal{C}$, and similar argument as in the proof of (i) would lead to a contradiction. In turn, if $\{u_n\}$ is bounded on $\gamma,$  we can find a triangle $\Delta$ with edges $a_1,a_2$ and $a_3$ so that $a_1\subset\gamma$, $a_2\subset \mathcal{D}\cap C_i$ and $a_3$ lies in $\mathcal{D}$, again, this leads to a contradiction and hence finishes the proof of (iv).
\end{proof}

In particular, this proposition implies the next result.

\begin{Proposition}\label{Real struct.}
Let $\Omega\subset\mathbb{P}$ be an admissible domain whose boundary is a union of f-convex arcs $C_i$. Let $\{u_n\}$ be either an increasing or a decreasing sequence of solutions to \eqref{soliton} over $\Omega$ such that for every open arc $C_i$, $u_n$ extends continuously to $C_i$ and either $u_n|_{C_i}$ converge uniformly to a continuous function or $+\infty$ or $-\infty,$ respectively. Let $\mathcal{D}$ be the divergence set of $\{u_n\}.$ Then each connected component of $\mathcal{D}$ is an admissible polygon in $\Omega.$
\end{Proposition}

\begin{remark}
We refer the reader to the classical reference \cite{Jenkins-Serrin} for further details on the proof of the proposition above in the Euclidean space, minor modifications of which yield the proof in our setting.
\end{remark}

\section{Existence of Jenkins-Serrin graphs}\label{E-J-S}
We have now obtained the local existence and studied the properties of sequences of solutions to \eqref{soliton}, and in this section we will finally prove the existence of Jenkins-Serrin solution of \eqref{soliton}. Before stating the main result, we recall and introduce some  notations. From now on, $\Omega$ always will be an admissible domain in $\mathbb{P}$ so that 
	\[
		\p\Omega = \left( \bigcup_{i=1}^l A_i \right) \bigcup \left( \bigcup_{j=1}^tB_j \right)	
		\bigcup \left( \bigcup_{k=1}^z C_k \right),
	\]
where the arcs $A_i$ and $B_j$ are $f$-geodesics and the arcs $C_k$ are $f$-convex. Let $\mathcal{P}$ be an admissible polygon. Then with the notations above, we define
\[
\alpha_f(\mathcal{P})=\sum_{A_i\subset\p\mathcal{P}}{\rm L}_f[A_i] \quad \text{and} \quad \beta_f(\mathcal{P})=\sum_{B_i\subset\p\mathcal{P}}{\rm L}_f[B_i]. 
\]
Recall that a function $u \colon \Omega\to\R$ is a Jenkins-Serrin solution of \eqref{soliton} over $\Omega$ with continuous boundary data $c_k \colon C_k\to\R$ if $u$ is a solution of \eqref{soliton} such that $u=c_k$ on $C_k$ for all $k$, $u\to+\infty$ on $A_i$ for all $i$, and $u\to-\infty$ on $B_j$ for all $j$. If $\{C_k\}=\varnothing,$ then we only require that $u\to+\infty$ on $A_i$ for all $i$ and $u\to-\infty$ on $B_j$ for all $j.$ 
\begin{Theorem}[Existence]\label{Existence}
Let $\Omega\subset \mathbb{P}$ be an admissible domain such that for any admissible polygon $\mathcal{P} \subset \bar\Omega$ we have
\begin{equation}\label{struc. cond.}
	2\alpha_f(\mathcal{P}) < {\rm L}_f[\p\mathcal{P}] 
		\quad\text{and} \quad 
	2\beta_f(\mathcal{P}) < {\rm L}_f[\p\mathcal{P}]. 
\end{equation}
Then
\begin{itemize}
\item[(a)] If $\{C_k\}\neq\varnothing$ and $c_k \colon C_k\to\R$ are  continuous functions {\rm(}bounded from below, in case $\{B_j\}=\varnothing$ and bounded from above, in case $\{A_j\}=\varnothing${\rm)}, then there exists a Jenkins-Serrin solution of \eqref{soliton} with continuous boundary data $c_k$.

\item[(b)] If $\{C_k\}=\varnothing$ and $\alpha_f(\Omega)=\beta_f(\Omega),$ then there exists a Jenkins-Serrin solution of \eqref{soliton}. 
\end{itemize}
Converselly, if $u$ is a Jenkins-Serrin solution of \eqref{soliton} with  continuous boundary data $$c_k \colon C_k\to\R$$ and if $\{C_k\}\neq\varnothing,$ then inequalities  \eqref{struc. cond.} hold for all admissible polygons $\mathcal{P}$, and if $\{C_k\}=\varnothing$ then we also have $\alpha_f(\Omega)=\beta_f(\Omega)$.
\end{Theorem}

\begin{proof}
The  proof will be divided  into three cases depending on the structure of $\partial\Omega$.
\begin{flushleft}
\textbf{1st Case}: Assume that $\{B_j\}=\varnothing$ and each function $c_k$ is continuous and bounded from below.
\end{flushleft}
By Theorem \ref{Perron's method} and Theorem \ref{Perron's method-boundary data} there exists a solution $u_n$ of \eqref{soliton} satisfying $u_n|A_i=n$ and $u_n|C_k=\min\{n,c_k\}.$ Moreover, by Theorem \ref{Max. Principle} the sequence $\{u_n\}$ is increasing. Let $\mathcal{D}$ be the divergence set of $\{u_n\}.$ If $\mathcal{D}\neq\varnothing,$ then by Proposition \ref{Real struct.} each connected component of $\mathcal{D}$ is an admissible polygon to $\Omega.$ Taking any polygon $\mathcal{P}\subset\mathcal{D}$ and using Lemma \ref{Flux-1} and Lemma \ref{Flux-2} we conclude that
\begin{equation*}
0={\rm F}_{u_n}[\p\mathcal{P}]=\sum_{A_i\subset\p\mathcal{P}}F_{u_n}[A_i]+F_{u_n}\left[\p\mathcal{P}\setminus\bigcup_{A_i\subset\p\mathcal{P}}A_i\right],
\end{equation*}
\begin{equation*}
\left|\sum_{A_i\subset\p\mathcal{P}}F_{u_n}[A_i]\right|\leq\alpha_f(\mathcal{P})
\end{equation*}
and 
\begin{equation}\label{mainthm-eq1}
\lim_n{\rm F}_{u_n}\left[\p\mathcal{P}\setminus\bigcup_{A_i\subset\p\mathcal{P}}A_i\right]=-{\rm L}_f\left[\p\mathcal{P}\setminus\bigcup_{A_i\subset\p\mathcal{P}}A_i\right]=-{\rm L}_f[\p\mathcal{P}]+\alpha_f(\mathcal{P}),
\end{equation}
where the first equality in \eqref{mainthm-eq1} holds due to the argument that we used to prove assertion (i) in Proposition \ref{divergence}. This would imply ${\rm L}_f[\p\mathcal{P}]\leq2\alpha_f(\mathcal{P})$, which is a contradiction, and therefore we must have $\mathcal{D}=\varnothing$. Now by Proposition \ref{Compactness Theorem} and Proposition \ref{Continuous data} a subsequence of $\{u_n\}$ converges uniformly on compact subsets of $\Omega$ to a solution $u$ of \eqref{soliton} with the required properties.\\

Now we prove that the existence of a solution implies the structural conditions \eqref{struc. cond.}. For this, suppose that $u \colon \Omega\to\R$ is a Jenkins-Serrin solution of \eqref{soliton} with boundary data $c_k \colon C_k\to\R$, where $c_k$ is continuous and bounded from below. Take any admissible polygon $\mathcal{P}$ in $\Omega.$ By Lemma \ref{Flux-1} we have
\begin{align*}
\alpha_f(\mathcal{P}) &= {\rm F}_u\left[\bigcup_{A_i\subset\p\mathcal{P}}A_i\right]=-{\rm F}_u\left[\p\mathcal{P}\setminus\bigcup_{A_i\subset\p\mathcal{P}}A_i\right] \\
&< {\rm L}_f\left[\p\mathcal{P}\setminus\bigcup_{A_i\subset\p\mathcal{P}}A_i\right]
= {\rm L}_f[\p\mathcal{P}]-\alpha_f(\mathcal{P}),
\end{align*}
since there exists at least one arc $\eta$ of $\p\mathcal{P}$ so that either $\eta$ lies in $\Omega$ or $\eta$ coincides with an arc $C_k.$ Therefore $2\alpha_f(\mathcal{P})<{\rm L}_f[\p\mathcal{P}]$ for each admissible polygon $\mathcal{P}$ in $\Omega.$

\begin{flushleft}
\textbf{2nd Case}: Assume that $\{A_i\}\neq\varnothing$, $\{B_j\}\neq\varnothing$ and $\{C_k\}\neq\varnothing.$
\end{flushleft}
By the first case there exist solutions $u^+$ and $u^-$ of \eqref{soliton} so that 
\begin{equation*}
u^+\equiv0\ \operatorname{on}\ \{B_j\},\ u^+|C_k=\max\{0,c_k\}\ \operatorname{and}\ u^+\to+\infty\ \operatorname{on}\ \{A_i\}
\end{equation*}
and
\begin{equation*}
u^-\equiv0\ \operatorname{on}\ \{A_i\},\ u^-|C_k=\min\{0,c_k\}\ \operatorname{and}\ u^-\to-\infty\ \operatorname{on}\ \{B_j\}.
\end{equation*}
Moreover, by Proposition \ref{Perron's method} and Proposition \ref{Perron's method-boundary data} for each $n$ there exists a solution $u_n$ of \eqref{soliton} so that 
\begin{equation*}
	u_n \equiv n\ \operatorname{on}\ \{A_i\},\ 
	u_n|C_k=\tilde{c}_k\ \operatorname{and}\ 
	u_n\equiv -n\ \text{ on } \ \{B_j\},
\end{equation*}
where 
\[
\tilde{c}_k =
\begin{cases}
n, &\text{if }  c_k \geq n; \\
c_k, &\text{if }  -n\leq c_k\leq n;  \\
-n, &\text{if }  c_k \leq -n. 
\end{cases}
\]

Since $u^-\leq u_n\leq u^+$, by Theorem \ref{Max. Principle}, then by Proposition \ref{Compactness Theorem} and Proposition \ref{Continuous data} a subsequence of $\{u_n\}$ must converge uniformly on compact subsets of $\Omega$ to a solution $u$ of \eqref{soliton} with the required boundary data.

To conclude this case, we prove that the existence of a solution implies the structural conditions \eqref{struc. cond.}. Suppose that $u \colon \Omega\to\R$ is a Jenkins-Serrin solution with continuous boundary data $c_k \colon C_k\to\R$. Take any admissible polygon $\mathcal{P}$ in $\Omega.$ If $\mathcal{P}\neq\Omega$, then there exists an edge of $\p\mathcal{P}$ which lies in $\Omega$ and from Lemma \ref{Flux-1} we obtain
\begin{align*}
\beta_f(\mathcal{P}) &= -{\rm F}_u\left[\bigcup_{B_j\subset\p\mathcal{P}}B_j\right]={\rm F}_u\left[\p\mathcal{P}\setminus\bigcup_{B_j\subset\p\mathcal{P}}B_j\right] \\
&< {\rm L}_f\left[\p\mathcal{P}\setminus\bigcup_{B_j\subset\p\mathcal{P}}B_j\right] 
= {\rm L}_f[\p\mathcal{P}]-\beta_f(\mathcal{P}).
\end{align*}
Therefore $2\beta_f(\mathcal{P})<{\rm L}_f[\p\mathcal{P}]$, and by the first case, we have also $$2\alpha_f(\mathcal{P})<{\rm L}_f[\p\mathcal{P}]$$ for each admissible polygon $\mathcal{P}\neq\Omega$. As these conditions are satisfied also when $\mathcal{P}=\Omega$, we have proved the second case.

\begin{flushleft}
\textbf{3rd Case}: Assume that $\{C_k\}=\varnothing$.
\end{flushleft}
First notice that the hypothesis on $\Omega$ implies that $l=t$, i.e., there is equal number of arcs $A_i$ and $B_j$. For each $n$ let $v_n$ be the solution of \eqref{soliton} satisfying $v_n|A_i=n$ and $v_n|B_j=0$. Clearly by Theorem \ref{Max. Principle} we must have $0\leq v_n\leq n$. Given any $c\in(0,n)$, we denote 
	\[
	E_c \coloneqq \{ p \in \Omega \colon v_n(p)>c \} \quad \text{and} \quad F_c \coloneqq \{ p \in \Omega \colon v_n(p) < c \}.
	\]
Let $E^i_c$ be the connected component of $E_c$ whose closure contains $A_i$, and similarly let $F^j_c$ be connected component of $F_c$ whose closure contains $B_j$. Notice that if $E_c\neq\bigcup_iE^i_c,$ then $v_n$ is a constant by maximum principle. Hence $E_c = \bigcup_i E^i_c$, and similarly we conclude that $F_c = \bigcup_j F^j_c$.

Let now $c$ be so close to $n$ that $\{E_c^i\}$'s are pairwise disjoint. This is possible by our assumption on $\Omega$ and $u_n$. Define 
	\[
		\mu(n)=\inf \{ c \in (0,n) \colon E_c^i\cap E_c^j = \varnothing\ \text{ for all } \ i\neq j\}.
	\]
Since $\overline{\Omega}$ is compact, there exists at least one pair $i$ and $j$ so that
	\[
	\overline{E}_{\mu(n)}^i \cap \overline{E}_{\mu(n)}^j \neq \varnothing.
	\]
Moreover, for each $i$ there exists $j$ so that
	\[
	F_{\mu(n)}^i\cap F_{\mu(n)}^j=\varnothing,
	\]
because if this was not the case, then $\cup_i F_{\mu(n)}^i$ would be connected, and consequently $\overline{E}_{\mu(n)}^i \cap \overline{E}_{\mu(n)}^j = \varnothing$.
 
Now, for every $n$, we define the function $u_n=v_n-\mu(n)$, and we prove that $\{u_n\}$ is locally bounded on compact subsets of $\Omega.$ To do this, we note that by the first case there exist auxiliary functions $u_i^+$ and $u_i^-$ that satisfy
\begin{equation*}
	u_{i}^{+} \equiv 0 \ \text{ on } \ \p\Omega \setminus A_i, \quad u_{i}^{+}|_{A_i} = +\infty
\end{equation*}
and
\begin{equation*}
u_i^-|_{B_j} \equiv -\infty \ \text{ for } \ j\neq i,\ \operatorname{and}\ u_i^-=0\ \operatorname{on}\ \p\Omega\setminus\bigcup_{j\neq i}B_j.
\end{equation*}
Then, given any $p\in\Omega$, we define the functions
	\[
		u^+(p) \coloneqq \max_{i} \{u^+_i(p)\} \ \text{ and } \ u^-(p) \coloneqq \min_{i}\{u^-_i(p) \},
	\]
and claim that 
	\[
	u^-\leq u_n\leq u^+
	\]
holds in $\Omega$.

Let $p\in\Omega$, and note first that if $u_n(p)=0$, then we have the claim. Therefore we suppose that $u_n(p) > 0$, which implies that $v_n(p)>\mu(n),$ and consequently we must have $p\in E_{\mu(n)}^i.$ Since $u_n\leq u_i^+$ on $\p E_{\mu(n)}^i,$ then by Theorem \ref{Max. Principle} we must have $u_n\leq u_i^+\leq u^+$ in $E_{\mu(n)}^i$. As $u^-$ is negative, we have the desired inequality $u_n(p)>0$. Finally, if $u_n(p)< 0$ we can apply the same argument replacing $E_{\mu(n)}^i$ by $F_{\mu(n)}^i$. 
Therefore $\{u_n\}$ is locally bounded on compact subsets of $\Omega.$

By construction
	\[
		u_n|A_i=n-\mu(n) \quad \text{and} \quad u_n|B_j=-\mu(n),
	\]
and to finish the proof, we show that $\{n-\mu(n)\}$ and $\{\mu(n)\}$ are diverging to infinity. Then we would have that a subsequence of $\{u_n\}$ converges uniformly on compact subset of $\Omega$ to a solution $u$ of \eqref{soliton} with the desired properties. We show that $\{n-\mu(n)\}$ diverges, and similar argument proves the claim also for $\{\mu(n)\}$. On the contrary, suppose that there exists a subsequence of $\{n-\mu(n)\}$ converging to a finite limit $\tau$. This implies that $\mu(n)\to+\infty$ and hence
	\[
		u_n = n - \mu(n)\to\tau\ \text{ on } \ A_i \ \text{ and } \ u_n = -\mu(n)\to-\infty \ 	
		\text{ on } \ B_j.
	\]
Let $u$ be the solution obtained from a convergent subsequence of $\{u_n\}$, so that
	\[
		u\to\tau \ \text{ on } \ A_i \quad \text{and} \quad u_n\to-\infty \ \text{ on } \ B_j.
	\]
From Lemma \ref{Flux-1} we have
	\begin{equation*}
		0 = F_u[\p\Omega] = F_u \left[ \bigcup_{i} A_i \right] + F_u \left[\bigcup_{j} B_j \right],
	\end{equation*}
but, on the other hand, Lemma \ref{Flux-1} gives
	\begin{equation*}
		\left|F_u\left[\bigcup_{i}A_i\right]\right|<\alpha_f(\Omega) \quad \text{and} \quad 
		F_u\left[\bigcup_{j}B_j\right]=-\beta_f(\Omega),
	\end{equation*}
which is a contradiction with our hypothesis on $\Omega$. Consequently $\{n-\mu(n)\}$ is diverging to infinity. 

Recall that we proved in the previous case that the existence of Jenkins-Serrin solution implies the structural conditions \eqref{struc. cond.} for each admissible polygon $\mathcal{P}\neq\Omega$. Therefore it remains to prove the last structural condition when $\mathcal{P}=\Omega.$ But the last condition follows now by Lemma \ref{Flux-1}, since
\begin{align*}
\beta_f(\Omega)=-{\rm F}_u\left[\bigcup_{j}B_j\right]={\rm F}_u\left[\p\Omega\setminus\bigcup_{j}B_j\right] &= {\rm F}_u\left[\bigcup_{i}A_i\right]=\alpha_f(\Omega).
\end{align*}
\end{proof}

We also have the following uniqueness result for the Jenkins-Serrin graphs.

\begin{Theorem}[Uniqueness]\label{thm-uniq}
Let $\Omega \subset \mathbb{P}$ be a bounded admissible domain and suppose that $u_1$ and $u_2$ are solutions of \eqref{soliton}. Then, if $\{C_k \} \neq \varnothing$ and $u_1=u_2$ on $\{C_k \}$, we have $u_1=u_2$ in $\Omega$. On the other hand, if $\{C_k \} = \varnothing$, then $u_2-u_1$ is a constant.
\end{Theorem}
\begin{proof} Define a function
	\[
		\varphi\coloneqq
		\begin{cases}
			K, &\text{if } u_1-u_2\geq K; \\
			u_1-u_2, &\text{if } -K<u_1-u_2\leq K;  \\
			-K, &\text{if } u_1-u_2\leq-K,
		\end{cases}
	\]
where $K$ is a large constant.
Then, as before, $\varphi$ is a Lipschitz function such that $-K\leq\varphi\leq K$, $\nabla\varphi = \nabla u_1-\nabla u_2$ in the set $\{ x\in\Omega \colon -K <u_1(x) - u_2(x) <K \}$ and $\nabla\varphi=0$ almost everywhere is the complement of $\{ x\in\Omega \colon -K <u_1(x) - u_2(x) <K \}$. Let 
	\[
	\Omega_{\epsilon,\delta} \coloneqq \{ x \in \Omega \colon \operatorname{dist}(x, \p\Omega) \geq \epsilon \} \setminus \bigcup_{p \in \Upsilon} B_{\delta}(p),
	\]
where $\epsilon,\delta > 0$ are small constants and $\Upsilon$ denotes the set of endpoints of $A_i$ and $B_j.$ Define also a function
\begin{equation}\label{Equation J*}
	J \coloneqq \int_{\p\Omega_{\epsilon,\delta}} \varphi \left[ h_c\left(f^2\frac{\nabla u_1}{W_1},\nu\right)-h_c\left(f^2\frac{\nabla u_2}{W_2},\nu\right) \right],
\end{equation}
where $\nu$ denotes the outer unit normal to $\p\Omega_{\epsilon,\delta}$. Since $\varphi$ is a Lipschitz function, the  divergence theorem and \eqref{Ineq. normal} give
\begin{align}\label{Ineq. Under J*}
	\nonumber J &= \int_{\Omega_{\epsilon,\delta}} h_c \left( \nabla\varphi, f^2\frac{\nabla u_1}{W_1}-f^2\frac{\nabla u_2}{W_2} \right)  \\
	&= \int_{\Omega_{\epsilon,\delta}}\frac{1}{2}(W_1+W_2)g_c(N_1-N_2,N_1-N_2),
\end{align}
where $N_i = \frac{\partial_s}{fW_i}-f\frac{\nabla u_i}{W_i}.$ 

On the other hand, observe that the boundary $\p\Omega_{\epsilon,\delta}$ is formed by arcs $A_i'$, $B_j'$, $C_k'$ and parts of $\partial B_{\delta}(p)$ when $p$ moves along $\Upsilon$. Here $A_i' = \partial \Omega_{\epsilon,\delta} \cap \{ x\in \Omega \colon \operatorname{dist}(x,A_i) \le \epsilon \}$ and similarly for $B_j'$ and $C_k'$.

 Next we suppose that $\epsilon<\delta$ and define 
 	\[
	 \Gamma = \p\Omega_{\epsilon,\delta} \setminus \bigcup_{i} A_i' \bigcup_{j} B_j' \bigcup_{k}C_k'.
	 \]
With this notation we obtain
\begin{align}\label{J* estimate}
J &= \int_{\Gamma} \varphi \left[ h_c \left(f^2\frac{\nabla u_1}{W_1},\nu\right)-h_c\left(f^2\frac{\nabla u_2}{W_2},\nu \right) \right] \\
\nonumber &+ \int_{\bigcup_i A_i'} \varphi \left[ h_c \left(f^2\frac{\nabla u_1}{W_1},\nu\right) - h_c \left(f^2\frac{\nabla u_2}{W_2},\nu \right) \right] \\
\nonumber &+ \int_{\bigcup_j B_j'} \varphi \left[ h_c \left(f^2\frac{\nabla u_1}{W_1},\nu\right) - h_c \left(f^2\frac{\nabla u_2}{W_2},\nu \right) \right] \\
\nonumber &+ \int_{\bigcup_k C_k'} \varphi \left[ h_c \left(f^2\frac{\nabla u_1}{W_1}, \nu\right) - h_c \left(f^2\frac{\nabla u_2}{W_2}, \nu \right) \right].
\end{align}

Since $\varphi=0$ in $\{C_i\}$ if $\delta$ is small enough, the first and the last terms of \eqref{J* estimate} can be estimated as
\begin{equation}\label{1st estimate}
\left| \int_{\Gamma} \varphi \left[ h_c \left(f^2\frac{\nabla u_1}{W_1},\nu\right) - h_c \left(f^2\frac{\nabla u_2}{W_2}, \nu \right) \right] \right| \leq 2K \sum_{p\in\Upsilon} {\rm L}_f [\p B_{\delta}(p)]
\end{equation}
and
\begin{equation}\label{last estimate}
\left|\int_{\bigcup_k C_k'} \varphi \left[ h_c\left(f^2\frac{\nabla u_1}{W_1}, \nu \right)-h_c \left(f^2 \frac{\nabla u_2}{W_2}, \nu \right) \right] \right| \leq 2 \epsilon \sum_{k}{\rm L}_f[C_k].
\end{equation}
Regarding the second and third term of \eqref{J* estimate}, note that the arcs $A_i'$ and $B_j'$  are $\epsilon$-close to $A_i$ and $B_j$, respectively. By remark \ref{estimate of normal}, if $\epsilon$ is small enough, 
	\begin{equation*}
		1 \geq h_c \left( f\frac{\nabla u_i}{W_i}, \nu\right) \geq 1-\delta \ \text{ on } \ \gamma, \ \text{ if } \ u \to + \infty \ \text{ along } \gamma' \text{ and} \dist_{H}(\gamma,\gamma')<\epsilon
\end{equation*}
and
	\begin{equation*}
	-1 \leq h_c \left( f\frac{\nabla u_i}{W_i},\nu \right) \leq -1+ \delta \ \text{ on } \ \gamma, \ \text{ if } \ u \to -\infty \ \text{ along } \gamma' \text{ and} \dist_H(\gamma,\gamma')<\epsilon,
\end{equation*}
where $\gamma'$ is an arc of $\p\Omega$ and $\dist_H$ denotes the Hausdorff distance. In particular, these inequalities yield
	\begin{align}\label{Another estimate}
	\left| \int_{A_i'} \varphi \left[ h_c \left( f^2 \frac{\nabla u_1}{W_1}, \nu \right) - h_c \left( f^2 \frac{\nabla u_2}{W_2}, \nu \right) \right] \right| 
	\leq K\delta{\rm L}_f[A_i']
	\end{align}
and
	\begin{align}\label{Another estimate2}
		\left| \int_{B_j'} \varphi \left[ h_c \left( f^2 \frac{\nabla u_1}{W_1}, \nu \right) - h_c \left( f^2 \frac{\nabla u_2}{W_2}, \nu \right) \right] \right| 
\leq K \delta {\rm L}_f[B_j'].
\end{align}
Finally from \eqref{Ineq. Under J*}, \eqref{J* estimate}, \eqref{1st estimate}, \eqref{last estimate}, \eqref{Another estimate} and \eqref{Another estimate2} we have
\begin{align*}
	\int_{\Omega_{\epsilon,\delta}}\frac{1}{2}(W_1+W_2)g_c(N_1-N_2,N_1-N_2)&\leq 2\epsilon\sum_{i}{\rm L}_f[C_i]+2K\sum_{p\in\Upsilon}{\rm L}_f[\p B_{\delta}(p)] \\
	&+ \sum_iK\delta{\rm L}_f[A_i']+\sum_jK\delta{\rm L}_f[B_j'].
\end{align*}

Taking $\delta\to 0$ we conclude that $N_1=N_2$ in the set $\{-K<u_1-u_2<K\}$ and hence $\nabla u_1=\nabla u_2$ in $\{-K<u_1-u_2<K\}.$ Since $K$ was arbitrary, $u_1 = u_2 + c$ in $\Omega$, where $c$ is a constant. Finally, if $\{C_i\} \neq \varnothing$ we must have $c=0$. 
\end{proof}


\section{Examples of Jenkins-Serrin graphs in $\R^3$ and $\mathbb{H}^2\times\R$}\label{Examples}
We finish this paper by giving some examples of domains that satisfy \eqref{struc. cond.} in $\R^3$ and in $\mathbb{H}^2\times\R$.

\subsection{Examples in $\R^3$}\label{EX. R}

In this case $\mathbb{P}$ is a vertical plane ($\R^2$) containing the vector $e_3$ in $\R^3$ and the Ilmanen's metric is given by $g_c=e^{cx_3}\langle \cdot,\cdot\rangle$, where $\langle \cdot,\cdot\rangle$ denotes the Euclidean metric of $\R^3$. Therefore the function $f$ is given by $f=e^{c\frac{x_3}{2}}$, and from \eqref{f-geodesic curvature} we see that $\gamma$ is an $f$-geodesic in $\mathbb{P}$ if and only if $\gamma$ satisfies 
\[
{\rm k}[\gamma]=c\langle N, e_3 \rangle,
\] 
where ${\rm k}[\gamma]$ denotes the scalar curvature of $\gamma$ in $\mathbb{P}$, $N$ denotes the unit normal to $\gamma$ and $\langle \cdot,\cdot\rangle$ is the Euclidean metric of  $\mathbb{P}=\R^2$. In particular, $f$-geodesics are translating curves in $\R^2$.

Let us assume now that $c=1$. It is well known that the unique translating curves are vertical lines in the direction $e_3$ and the grim reaper curves, given by
	\[
	x_3 = -\log \cos x_1, \quad x_1\in(-\pi/2,\pi/2). 
	\]
Therefore we can produce admissible domains $\Omega\subset\mathbb{P}$ that are bounded by vertical line segments and parts of the grim reaper curves, see Figure \ref{fig-1}. If we assign boundary data $+\infty$ on the parts of the grim reaper curve (coresponding to the edges $A_1,A_2$ in Theorem \ref{Existence}) and continuous data ($0$ in Fig. \ref{fig-1}) on the vertical segments (corresponding to the edges $C_1,C_2$), the condition for the existence of solutions becomes
	\[
		{\rm L}_f[A_1] + {\rm L}_f[A_2] < {\rm L}_f[C_1] + {\rm L}_f[C_2].
	\]
\begin{center}
	\includegraphics[scale=0.9]{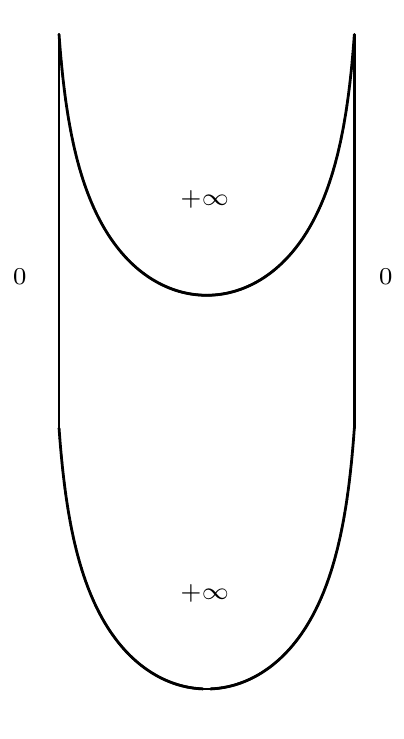}
	\captionof{figure}{Basic solution.\label{fig-1}}
\end{center}

Let us take the following parametrizations 
	\begin{align*}
		A_1&=\{(x_1, 0, a - \log \cos x_1) \colon x_1 \in (r,s)\}, \\ 
		A_2 &=\{(x_1, 0, b - \log \cos x_1) \colon x_1 \in (r,s)\},  \\
		C_1 &= \{(r,0,x_3) \colon x_3\in(a - \log \cos s , b - \log \cos s)\}, \ \text{ and } \\
		C_2 &= \{(s,0,x_3) \colon x_3\in(a - \log \cos r , b - \log \cos r)\}
	\end{align*} 
for the edges of $\Omega$ in the plane $\mathbb{P}\subset \R^3$, where $-\pi/2 < s < r < \pi/2$, $a,b \in\R$ and $a < b$. 
Then we have 
\begin{align}\label{length dom.}
	{\rm L}_f[A_1]+{\rm L}_f[A_2] 
		&= (e^{b} + e^{a})(\tan r - \tan s)\\
	\nonumber {\rm L}_f[C_1] + {\rm L}_f[C_2] 
		&= (e^{b}-e^{a})(\sec r + \sec s ).
\end{align}

If  we fix $a < b$, then choosing $r-s>0$ small enough, we ensure that ${\rm L}_f[A_1] + {\rm L}_f[A_2] < {\rm L}_f[C_1] + {\rm L}_f[C_2]$. In fact, if we reflect $n$ times the basic solution given by Figure \ref{fig-1}, we obtain a periodic surface with alternating boundary values $+\infty$ and $-\infty$, see Figure \ref{fig-2}.
\begin{center}\includegraphics[scale=0.9]{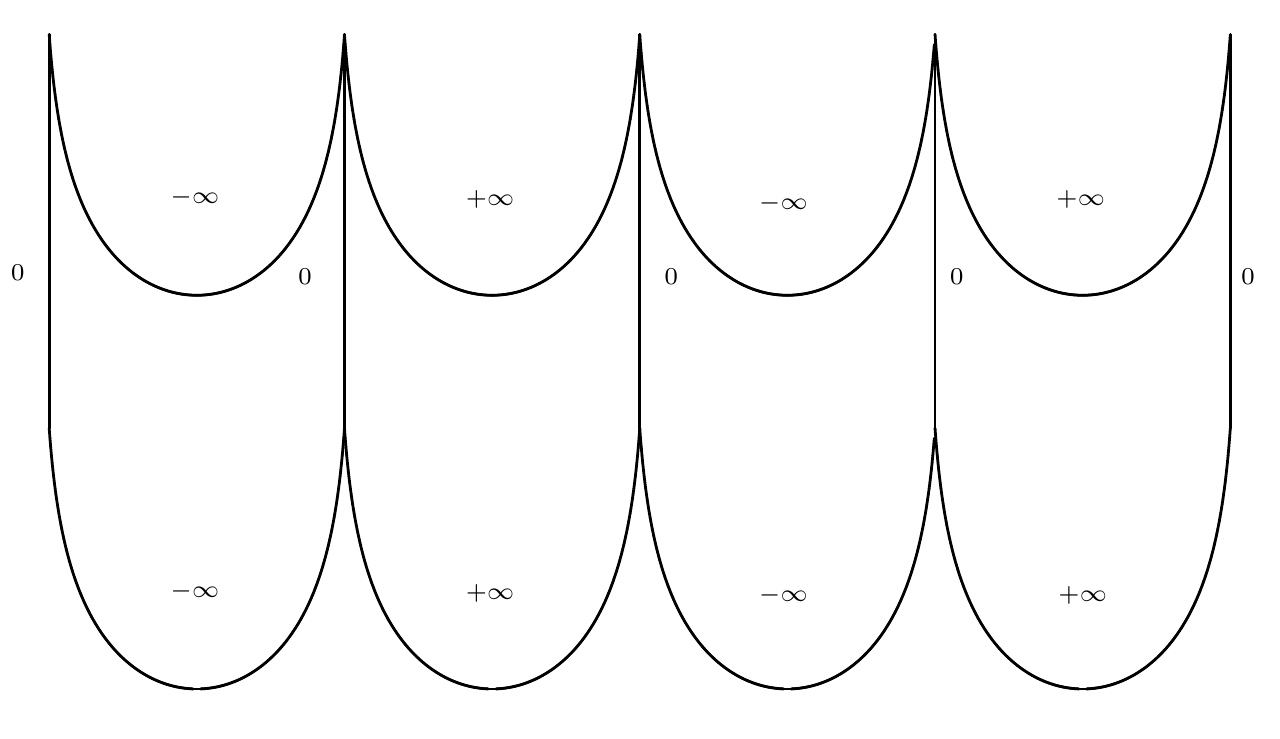}
\captionof{figure}{Reflections of the basic solution.\label{fig-2}}
\end{center}

On the other hand, if $r>s$ are fixed, then choosing $b-a>0$ small enough in \eqref{length dom.}, we can guarantee that ${\rm L}_f[A_1] + {\rm L}_f[A_2] > {\rm L}_f[C_1] + {\rm L}_f[C_2]$. In particular, if we rename $C_i$ by $B_i$, there are $b>a$ and $r>s$ so that ${\rm L}_f[A_1] + {\rm L}_f[A_2] = {\rm L}_f[B_1] + {\rm L}_f[B_2]$, and we obtain the structural condition of case (b) in  Theorem \ref{Existence}.


\subsection{Examples in $\mathbb{H}^2\times\R$}

Consider the hyperbolic plane $\h^2$ as a warped product $\h^2 = \R\times_{e^{x}}\R$ with the metric 
\begin{equation}\label{hyp-metr}
dx^2 + e^{2x}ds^2.
\end{equation} 
Then the vector field $\partial_s$ is a Killing field with norm $|\partial_s|_{(x,s)} = e^{x}$, and the $x$-axis is an integral curve of the distribution orthogonal to $\partial_s$. In this case we can take the vertical plane $\mathbb{P}$ in $\h^2\times\R$ to be the vertical plane over $x$-axis 
	\[
		\mathbb{P}=\{(x,t,s) \colon x,t\in\R,\, s=0\},
	\]
and with this choice we have $f=e^{c\frac{t}{2}}e^x$. Recall that we are endowing $\mathbb{P}$ with metric $h_c=e^{ct}(dx^2+dt^2)$. Furthermore, by \eqref{f-geodesic curvature} we have that $\sigma$ is a $f$-geodesic if 
	\[
		{\rm k}_{h_c}[\sigma]=h_c\left(\frac{\nabla f}{f}, N_\sigma\right),
	\] 
where $N_\sigma$ is the unit normal along $\sigma$ and $\nabla f$ is taken with respect to the metric $h_c$ in $\mathbb{P}.$ Since 

\begin{align*}
\frac{\nabla f}{f} &= e^{-ct}\left(\frac{c}{2}\p_t+\p_x\right), \\
{\rm k}_{h_c}[\sigma] &= e^{-ct/2}{\rm k}_{g_0}[\sigma] - e^{ct/2} g_0\left(\frac{c}{2}\partial_t, N\right), 
\end{align*}
and  $N_\sigma=e^{-ct/2}N$,
where $N$ denotes the unit normal along $\sigma$ with respect to $g_0=dx^2+dt^2$ and ${\rm k}_{g_0}[\sigma]$ denotes the scalar geodesic curvature of $\sigma$ with respect to $g_0,$ we have
\begin{equation}\label{f-geodesic in HxR}
	{\rm k}_{g_0}[\sigma] = g_0\left(c\partial_t+\p_x, N \right),
\end{equation}
From this equality we can conclude that lines in the direction $c\p_t+\p_x$ are $f$-geodesics in $\mathbb{P}$. 

To compute the other $f$-geodesics, let us denote $\vec{\tau}=\p_x+c\partial_t$ and $\vec{\varsigma}=c\partial_x-\p_t$ and notice that $\{\vec{\varsigma},\vec{\tau}\}$ is a positive frame of $\mathbb{P}.$ Assume that $\sigma(x)=x\vec{\varsigma}+\varphi(x)\vec{\tau},$ where $x\in\R$ and $\varphi$ is a smooth function. From \eqref{f-geodesic in HxR} we can conclude that 
\[
	\frac{\varphi''}{1+(\varphi')^2}=|\vec{\tau}|^2.
\]
Consequently $\varphi(x)=-|\vec{\tau}|^{-2}\log\cos(|\vec{\tau}|^{2}x)$ for $x\in\left(-\pi/(2|\vec{\tau}|^{2}),\pi/(2|\vec{\tau}|^{2})\right)$. Using translation of $\sigma$ we can conclude that $f$-geodesics of $\mathbb{P}$ are either lines in the direction of $\vec{\tau}$ or translating the curve $\sigma$ above, which is the grim reaper curve in the direction of $\vec{\tau}.$
Finally, the argument of the subsection \ref{EX. R} allows us to conclude the existence of similar basic domains. 

\begin{remark}
We observe that the induced metric $h_c$ on $\mathbb{P}$ is exactly the same as the metric $g_c$ in Subsection \ref{EX. R}. Therefore, if we use coordinates $(x,t)$ in $\mathbb{P}$ instead of $(x_1, x_3)$ then the admissible domains are exactly the same in both spaces. 
\end{remark}

Notice also that since $\h^{2}=\R\times_{e^{x}}\R$, by Remark \eqref{f-prop.}, we can conclude the following.
\begin{Proposition}
The hypersurface $\sigma\times\R$ is a properly embedded translating soliton in $\mathbb{H}^2\times\R$ with respect to $\p_t$. Moreover, $\alpha\times\R$ is a properly embedded translating soliton in $\mathbb{H}^2\times\R,$ when $\alpha$ is any line in the direction $\vec{\tau}=\p_x+c\partial_t.$
\end{Proposition}

\begin{Remark}
The example $\alpha \times \mathbb{R}$ was already known in \cite{Lira}, but $\sigma \times \mathbb{R}$ is a new example of a properly embedded translating soliton in $\mathbb{H}^2\times\R.$
\end{Remark}

\bibliographystyle{amsplain, amsalpha}

\end{document}